\documentclass[10pt]{amsart}

\usepackage{dsfont, amsmath, amsfonts, amsthm, amssymb, times, mathtools}
\usepackage[euler-digits]{eulervm}
\numberwithin{equation}{section}

\usepackage{hyperref}
\hypersetup{
    pdftoolbar=true,
    pdfmenubar=true,
    pdffitwindow=false,
    pdfstartview={FitH},
    pdftitle={Multiple orthogonal geodesic chords and a proof of Seifert's conjecture on brake orbits },
    pdfauthor={Roberto Giamb\`o, Fabio Giannoni and Paolo Piccione},
    pdfsubject={A proof of Seifert conjecture on the brake orbits},
    pdfkeywords={},
    pdfnewwindow=true,
    colorlinks=true, 
    linkcolor=blue,
    citecolor=blue,
    urlcolor=blue,
}

\makeatletter
\g@addto@macro\bfseries{\boldmath}
\makeatother

\newcommand{\dist}{\operatorname{dist}}
\newcommand{\cat}{\operatorname{cat}}
\newcommand{\Rcal}{\mathcal R}

\newcommand{\Ncal}{\mathcal N}

\newcommand{\Dcal}{\mathcal D}

\newcommand{\Dgot}{\mathfrak D}

\newcommand{\Ddt}{\tfrac{\mathrm D}{\mathrm dt}}
\newcommand{\Dds}{\tfrac{\mathrm D}{\mathrm ds}}

\theoremstyle{plain}\newtheorem{teo}{Theorem}[section]
\theoremstyle{plain}\newtheorem{prop}[teo]{Proposition}
\theoremstyle{plain}\newtheorem{lem}[teo]{Lemma}
\theoremstyle{plain}\newtheorem{cor}[teo]{Corollary}
\theoremstyle{definition}\newtheorem{defin}[teo]{Definition}
\theoremstyle{remark}\newtheorem{rem}[teo]{Remark}
\theoremstyle{definition}
\theoremstyle{remark}
\theoremstyle{plain}

\newtheorem*{mainthm}{\sc Theorem (Seifert's Conjecture on Brake Orbits)}

\title[Multiple OGCs and the Seifert's conjecture]{Multiple orthogonal geodesic chords\\ and a proof of Seifert's conjecture on brake orbits}

\author[R.\ Giamb\`o]{Roberto Giamb\`o$^{1,2}$}
\address{$^1$Scuola di Scienze e Tecnologie, Universit\`a di Camerino, Camerino (MC), Italy}
\address{$^2$INFN, Sezione di Perugia, Perugia, Italy}
\email{roberto.giambo@unicam.it}

\author[F.\ Giannoni]{Fabio Giannoni$^1$}
\email{fabio.giannoni@unicam.it}

\author[P. Piccione]{Paolo Piccione$^3$}
\address{$^3$Departamento de Matem\'atica , Universidade de S\~ao Paulo, S\~ao Paulo (SP), Brazil}
\email{piccione@ime.usp.br}

\date{June 3, 2022}


\begin{document}

\begin{abstract}
Using a pseudo-gradient approach and minimax theory, we prove the existence of at least $N$ orthogonal geodesic chords in a class of Riemannian $N$-disk with strongly concave boundary. This yields a proof of a celebrated conjecture by Seifert  \cite{seifert} on the number of brake orbits in a potential well of a natural Lagrangian/Hamiltonian system.
\end{abstract}

\maketitle

\renewcommand{\contentsline}[4]{\csname nuova#1\endcsname{#2}{#3}{#4}}
\newcommand{\nuovasection}[3]{\hbox to \hsize{\vbox{\advance\hsize by -1cm\baselineskip=10pt\parfillskip=0pt\leftskip=3.5cm\noindent\hskip -2.5cm \textsl{#1}\leaders\hbox{.}\hfil\hfil\par}$\,$#2\hfil}}
\newcommand{\nuovasubsection}[3]{\hbox to \hsize{\vbox{\advance\hsize by -1cm\baselineskip=10pt\parfillskip=0pt\leftskip=4cm\noindent\hskip -2cm #1\leaders\hbox{.}\hfil\hfil\par}$\,$#2\hfil}}


\begin{small}
\tableofcontents
\end{small}


\section{Introduction and statement of the results}\label{sec:intro}
\subsection*{Orthogonal geodesic chords}
Let $(M,g)$ be a  Riemannian manifold with $\mathrm{dim}(M)=N\ge2$ and let $\Omega\subset M$ be an open subset with smooth boundary $\partial\Omega$; $\overline\Omega=\Omega\bigcup\partial \Omega$ will denote its closure. The main objects of interest here are \emph{orthogonal geodesic chords in $\overline\Omega$}, \textbf{OGC}s for short, i.e. noncostant  geodesics $\gamma:[a,b] \to \overline\Omega$ that start and arrive orthogonally to $\partial\Omega$ and such that $\gamma\big(\left]a,b\right[\big) \subset \Omega$. Our aim is to determine a lower bound on the number of OGCs when $\overline\Omega$ is homemorphic to an $N$-disk and to use it to prove a conjecture due to H. Seifert (cf. \cite{seifert}).

The case when $\overline\Omega$ is convex is studied in a classical paper by Bos, see \cite{bos}. Bos' result says that, when $\overline\Omega$ is homeomorphic to an $N$-disk and convex, then there are at least $N$ distinct OGCs in $\overline\Omega$. Such a result is a generalization of a classical result by Lusternik and Schnirelman (see \cite{LustSchn}), where the same estimate was proven for convex subsets of $\mathds R^N$ endowed with the Euclidean metric.  
Recently, Bos' result has been extended to cases satisfying assumptions weaker than convexity, see \cite{OT}.

Note that convexity is an essential assumption for the use of curve shortening method. Namely, in this situation, geodesics in $\overline\Omega$ can touch $\partial\Omega$ only at their endpoints (or lie entirely on $\partial \Omega$), and shortening a curve in $\overline\Omega$ by broken geodesics produces a curve that remains inside $\overline\Omega$. 

When studying the non-convex case, classical variational approaches fail, and new phenomena need to be considered. For instance, one needs to take into account  the existence of geodesics in $\overline\Omega$ that are tangent to $\partial\Omega$. Arbitrarily small neighborhood of such geodesics, in the appropriate functional spaces, contain curves that leave $\Omega$ under shortening flows, hence standard gradient flows arguments do not work.

We define \emph{weak orthogonal geodesic chord}, \textbf{WOGC} for short, any nonconstant  geodesic chord $\gamma\colon[a,b]\to\overline\Omega$ starting from and arriving to $\partial\Omega$ orthogonally, and such that $\gamma(s)\in\partial\Omega$ for some $s\in\left]a,b\right[$. Although WOGCs may in principle exist, our method will only detect OGCs that touch the boundary only at their endpoints, relegating WOGCs to a secondary role with no contribution to the topological invariant (relative LS category) employed here.
\smallskip

\noindent\quad Besides the obvious geometrical appeal, the main interest in orthogonal geodesics chords in the non-convex case comes from classical dynamical systems. Maupertuis principle gives a bridge between solutions of a natural Lagrangian/Hamiltonian system having a fixed total energy value, with geodesics in configuration space endowed with a suitable conformal metric. In particular, the \emph{brake orbits} of the system, which form a special class of periodic solutions, correspond via Maupertuis principle to OGCs of the conformal metric in suitable open sets $\Omega_\delta$, whose closure is contained in the interior of the potential well and $\partial\Omega_\delta$ is close to the boundary. 
Such conformal metric  makes $\overline{\Omega_\delta}$   \emph{strongly concave}, i.e., with positive-definite second fundamental form in the exterior normal direction.
\subsection*{Brake orbits and Seifert's conjecture}
Let us illustrate briefly a Lagrangian formulation of the brake orbits problem (more details in subsection~\ref{sub:BOOGC}). An equivalent formulation can be given for Hamiltonian systems, via Legendre transform, but it will not be needed here.

Let $\widehat M$ be an  $N$-dimensional manifold with $\widehat M$ of class $C^3$  representing the configuration space of some dynamical systems and $\widehat g$ a Riemannian metric of class $C^2$. Let $V\colon \widehat M\to\mathds R$ be a $C^2$--function, representing the potential energy of some conservative force acting on the system.
One looks for periodic solutions $x\colon [0,T]\to\widehat M$ of the Lagrangian systems:
\begin{equation}\label{eq:lagrsystem}
\Ddt\dot x=-\nabla V(x),
\end{equation}
where $\tfrac{\mathrm D}{\mathrm dt}$ denotes the covariant derivative of the Levi--Civita connection of $\widehat g$ for vector fields along $x$, and $\nabla V$ is the gradient of $V$. Solutions of \eqref{eq:lagrsystem} satisfy the conservation law of the energy  $\frac12g(\dot x,\dot x)+V(x)=E$, where $E$ is a real constant called the \emph{energy} of the solution $x$. It is a classical problem to give estimate of the number of periodic solutions of \eqref{eq:lagrsystem} having a fixed value of the energy $E$. This problem has been, and still is, the main topic of a large amount of literature, also for more general autonomous Hamiltonian systems, 
see for instance \cite{LiuLong,LZ,Liu,Long,rab} and the references therein.
Among all periodic solutions of \eqref{eq:lagrsystem}, historical importance is given to a special class called \emph{brake orbits}; these are ``pendulum-like'' solutions, that oscillate with constant frequency along a trajectory that joins two distinct endpoints lying in $V^{-1}(E)$. 
\medskip

A very famous conjecture due to Seifert, see \cite{seifert}, originally formulated under analytic regularity assumptions, asserts  that, given a Lagrangian system as in \eqref{eq:lagrsystem}, if the sublevel $V^{-1}\big(\left]-\infty,E\right]\big)$ is homeomorphic to an $N$-disc and $E$ is a regular value for $V$, then there should exist at least $N$ geometrically distinct brake orbits.%
\footnote{Two brake orbits $q_1$ and $q_2$ are called \emph{geometrically distinct} if the sets $q_1(\mathds{R})$ and $q_2(\mathds{R})$ are distinct.}
This estimate is known to be sharp, i.e. there are examples of analytic Lagrangian systems having energy sublevels homeomorphic to an $N$-disk and admitting exactly $N$ geometrically distinct brake orbits.\footnote{%
Given constants $\lambda_1,\ldots,\lambda_N\in\mathds R^+\setminus\{0\}$, with $\lambda_i/\lambda_j\not\in\mathds Q$ when $i\ne j$, if one considers the potential $V=\sum\limits_{i=1}^N\lambda_i^2x_i^2$ in $\mathds R^N$, for every value $E>0$ there are exactly $N$ geometrically distinct periodic solutions  having energy $E$ of the corresponding Lagrangian system, and they are brake orbits.} To the present days, Seifert's conjecture  has been solved affirmatively in some cases, see for instance \cite{2bo, OT, LZ,LZZ,Z1,Z2,Z3}.
In particular, \cite{LZ} contains a proof of the Seifert conjecture for Euclidean metrics, when the potential is assumed even and convex. In \cite{2bo}, Seifert's conjecture is proved in the case $N=2$. In \cite{OT}, the conjecture is proved for perturbations of radial potentials.
When the $E$-sublevel $V^{-1}\big(\left]-\infty,E\right]\big)$ has the topology of the annulus, multiplicity of brake orbits is studied in \cite{arma} and \cite{JDE0}.   
\smallskip

The central result of the present paper (Theorem~\ref{thm:generale}) gives a lower bound on the number of orthogonal geodesics in Riemannian disks with strongly concave boundary, and satisfying a technical, nevertheless mild, additional geometric assumption. Namely, we will assume that there exists some point in $\Omega$ through which there exists no geodesic with both endpoints on $\partial\Omega$, and which is either tangent to $\partial\Omega$ at both endpoints, or tangent to $\partial\Omega$ at one endpoint and orthogonal to $\partial\Omega$ at the other (see \eqref{eq:nonsat}). 
Such assumption has a technical nature, that will be discussed later on, and it is possibly inessential for the validity of the result of Theorem~\ref{thm:generale}.
\smallskip

It is now a well established fact, see \cite{GGP1,cag}, that fixed energy brake orbits for the system \eqref{eq:lagrsystem}  
correspond to OGCs in a domain $\overline\Omega$ contained (and diffeomorphic to) the corresponding energy sublevel of the potential $V$. The metric 
in $\overline\Omega$, which is usually called the \emph{Jacobi metric}, is conformal to $g$, and it makes $\partial\Omega$ strongly concave.  
It is important to observe that the technical assumption \eqref{eq:nonsat} of Theorem~\ref{thm:generale} is satisfied by the Jacobi metric in the cases of interest (Proposition~\ref{prop:SGL}).
In view of these results, Theorem~\ref{thm:generale} yields a proof of  Seifert's conjecture:
\begin{mainthm}\label{thm:seifert} Let $E$ be a regular value of the potential $V$, and assume that the sublevel $V^{-1}\big(\left]-\infty,E\right]\big)$ is  homeomorphic to the $N$--dimensional disk.
Then, the Lagrangian system \eqref{eq:lagrsystem} admits at least $N$ geometrically distinct brake orbits of energy $E$. 
\end{mainthm}
The above is the final result which came from a series of papers on the multiplicity of brake orbits in a potential well homeomorphic to a ball, see \cite{JDE,2bo,OT}, where lower bounds for the number of brake orbits were obtained under additional assumptions. 

\subsection*{A brief overview of the proof}
Following along the lines of \cite{OT}, we will determine the OGCs in $\overline\Omega$ using topological methods and integral flows of suitable vector fields. More precisely, OGCs are determined as paths in a suitable functional space of curves, that are critical points of the geodesic action functional \emph{relatively} to appropriately defined admissible variations. Such variations are obtained as flow of (local) vector fields defined in the space of curves, that represent \emph{infinitesimal} admissible variation. We will define a class of infinitesimal admissible variations, denoted by $\mathcal V^-$, and a subclass of $\mathcal V^-$ denoted by $\mathcal V^+$, see Sections~\ref{sec:critPS} and \ref{sec:V+criticalcurves}. Roughly speaking, elements of $\mathcal V^-$ are obtained as vector fields $V_x$ along curves $x$ in $\overline\Omega$ with the property that $V_x(s)$ points \emph{inward} whenever $x(s)\in\partial\Omega$. Thus, the corresponding variation produces curves that remain in $\overline\Omega$. Elements of $\mathcal V^+$ satisfy the additional requirement of changing their direction near $\partial\Omega$: $V_x$ points inward when $x$ touches $\partial\Omega$, and outward when $x$ is at a certain (prescribed) small distance from $\partial\Omega$, see \eqref{eq:defV+(x)globale}.  

A path in $\Omega$ will be called $\mathcal V^-$-critical when it is fixed by the flow of every local vector field in the class $\mathcal V^-$. However, in the nonconvex case, there are $\mathcal V^-$-critical curves that are \emph{not} OGCs, but rather curves that belong to a more general class, introduced in \cite{MS} and called \emph{geodesics with obstacle}.
The set of $\mathcal V^-$-critical paths that are not OGCs, denoted by  $Z^-$ (see \eqref{eq:Z-sigma}), is described in Section~\ref{sub:obstpbst}.  Notably, the strong convavity assumption implies in particular that geodesics with obstacle that are orthogonal to $\partial\Omega$ at the endpoints and that have bound\-ed length, the contact set with $\partial\Omega$ consists of a uniformly bounded number of disjoint intervals and isolated points (Remark~\ref{rem:concavconseg}).

Geodesics with obstacle are, roughly speaking, curves having possibly low regularity (say, $C^1$, or more precisely $H^{2,\infty}$), that are made up by portions that either lie on the boundary $\partial\Omega$ (the \emph{contact set} with $\partial\Omega$), or that are geodesics segments contained in the interior $\Omega$. While geodesics with obstacle with prescribed boundary conditions can be found in any compact Riemannian manifold with smooth boundary, arbitrary Riemannian manifolds with boundary may not contain any \emph{true} orthogonal geodesic chord. A very elementary counterexample is depicted in Bos' paper \cite{bos}, who considers a simple triangular shaped region in $\mathds R^2$ with non-convex rounded corners. In this case, the lack of OGCs, i.e., segments orthogonal to the boundary at both endpoints, is immediately verified by inspection.

However, as already mentioned, in the special case of Jacobi metrics studied in the present paper, the existence of OGCs is proved by a construction that uses in an essentail way the property of strong concavity of the disk. 
Strong concavity has an important consequence, that will be exploited in our construction:
geodesics with both endpoints on the boundary of $\Omega$ cannot remain uniformly close to $\partial\Omega$ (see Remark~\ref{rem:1.4bis}). Note that this property does not hold in Bos' counterexample \cite{bos}.

Using the two classes $\mathcal V^-$ and $\mathcal V^+$ described above, we construct a global flow on the space of paths in $\overline\Omega$ with endpoints in $\partial\Omega$, which plays the role of the flow of a \emph{pseudo-gradient} vector field (see \cite{PalaisLS}) for the geodesic action functional.
 The pseudo-gradient field is constructed locally in distinct regions of the space of admissible paths whose mutual distance is strictly positive, and then made global using convex linear combinations. Reference \cite{PalaisLS} provides the basic tools for the globalization of local flows, using partitions of unity. The local and global constructions of the pseudo-gradient field are scattered throughout Sections~\ref{sec:critPS}, \ref{sec:V+criticalcurves} and \ref{sec:CVO}, see Proposition~\ref{thm:constrcampiV-}, Proposition~\ref{V-onglobal}, Proposition~\ref{thm:constrcampiV1-}, Proposition~\ref{V+onC*}, and Proposition~\ref{prop:etaglobale}.

One of the main technical parts of the paper consists in the construction of a special set $\Lambda_*$ of paths (see Section~\ref{sec:CVO}), that satisfies the following properties:
\begin{itemize}
\item $\Lambda_*$ is invariant by a pseudo-gradient flow (Proposition~\ref{prop:etaglobale}, item \eqref{itm:C*}), and the geodesic action functional is strictly decreasing along the inward pointing flow lines that are near the \emph{entrance of $\Lambda_*$}\footnote{%
Given a semi-group $(\phi_t)_{t\ge0}$ of homeomorphisms of a topological space $\mathcal X$, and given a subset $\mathcal Y\subset\mathcal X$ which is $\phi_t$-invariant for all $t$, the \emph{entrance} of $\mathcal Y$ is the  set of the $x \in Y$ such that there exists $\delta_x >0$ such that $\phi_t(x) \not\in Y$ for any $t \in ]-\delta_x,0[$.
In our concrete setting, the entrance set of $\Lambda_*$ is denoted by $\Gamma_*$ and it is defined in \eqref{eq:Gamma*}.} (see also 
Proposition~\ref{prop:etaglobale}, item \eqref{itm:G*});
\item $\Lambda^*$ contains in its interior all the geodesics with obstacle that are orthogonal to $\partial\Omega$ at both endpoints;
\item $\Lambda^*$ contains all the WOGCs;  
\item $\Lambda^*$  does \emph{not} contain true OGCs; 
\item $\Lambda_*$ is \emph{topologically trivial}, meaning that it can be continuously retracted to a set of curves lying entirely on $\partial\Omega$, through a retraction that fixes the constant curves (Proposition~\ref{prop:sublevels}). What this implies, roughly speaking, is that points in $Z^-$ are not counted in the minimax argument, and therefore the topological invariant (relative Lusternik--Schnirelman category) employed in the minimax argument gives a lower bound for true OGCs.
\end{itemize}
{The global flow is defined using homotopies associated to the infinitesimal variations in $\mathcal V^+$ in $\Lambda_*$, to the more general variations in $\mathcal V^-$ far from $\Lambda_*$, and with convex combinations near the entrance set of $\Lambda_*$.} The technical geometric assumption of Theorem~\ref{thm:generale} mentioned above is used to construct the continuous retraction of $\Lambda_*$ onto a set of curves lying in $\partial\Omega$.%
\smallskip

Once the pseudo-gradient flow and the set $\Lambda_*$ have been defined, the proof of Theorem~\ref{thm:generale} follows standard general ideas from minimax theory (Section~\ref{sec:CVO}).  By the above construction, the fixed points of our flow that lie outside of $\Lambda_*$ are OGCs. The minimax procedure detects a number of points fixed by the flow (outside $\Lambda_*$) which is greater than or equal to the Lusternik-Schnirelman relative category of a set (see \eqref{eq:CC0}) which has the topology of the quotient space $(\mathbb S^{N-1}\times\mathbb S^{N-1})/\mathcal R$, where $\mathcal R(A,B)=(B,A)$ (category relative to the diagonal of $\mathbb S^{N-1}\times\mathbb S^{N-1}$). Such number is equal to $N$, see \cite[Appendix~A]{esistenza} for the details of the computation. Since the set $\Lambda_*$ can be continuously retracted to a set consisting of curves lying in the boundary of $\Omega$,  geodesics with obstacle do not contribute to the count of those fixed points detected by minimax. This implies that the strongly concave Riemannian $N$-disk under consideration possesses at least $N$ distinct OGCs, proving our desired result.

\subsection*{A short guide to an amicable reading}
With the aim of helping the reader to a more amicable reading of the paper, we give here a short presentation of the objects that will appear in our construction of the main result, with a description of their role in the proof.
\smallskip

First, we define a set of trial paths for our variational problem, which is denoted by $\mathfrak M$, see \eqref{eq:defM}, that is simply the set of $H^1$-curves in $\overline\Omega$ with free endpoints in $\partial\Omega$. 
Inside $\mathfrak M$, we define the sets $\mathfrak C_0\subset\mathfrak C$, see \eqref{eq:CC0}. Morally, $\mathfrak C$ is the set of \emph{straight segments/chords} in $\overline\Omega$ with endpoints in $\partial\Omega$. 
The point here is that, in principle, we are only assuming that $\overline\Omega$ is only \emph{homemorphic}, and not necessarily diffeomorphic, to an $N$-disk. Thus, straight segments in principle do not have the desired $H^1$-regularity. An essential point is that $\mathfrak C$ must be invariant by the \emph{backwards reparameterization operator} on $\mathfrak M$, that is denoted by $\mathcal R$, see \eqref{eq:2.4}. Clearly, the set of OGCs is $\mathcal R$-invariant; Given an OGC $\gamma$, then $\mathcal R(\gamma)$ is not geometrically distinct from $\gamma$, which implies that any multiplicity result must be obtained in the quotient space $\widetilde{\mathfrak M}=\mathfrak M/\mathcal R$. Thus, all our constructions require appropriate $\mathcal R$-invariant functions/sets or $\mathcal R$-equivariant flows. Details for the construction of $\mathfrak C$ are found in Lemma~\ref{thm:corde}. The subset $\mathfrak C_0$ consists of elements of $\mathfrak C$ that have the same initial and endpoint, i.e., constant curves taking value on $\partial\Omega$; note that $\mathcal R$ is the identity on $\mathfrak C_0$.\smallskip

The set $\mathfrak M$ is the domain for the classical geodesic action functional $\mathcal F$, see \eqref{eq:enfunct}; we search for critical points of $\mathcal F$ in $\mathfrak M$ that touch $\partial\Omega$ only at the endpoints. Due to the lack of convexity of $\partial\Omega$, standard minimax methods applied to the flow of the gradient field of $\mathcal F$ do not work here.  We will define a pseudo-gradient for $\mathcal F$, i.e., a \emph{(future) complete} vector field on $\mathfrak M$ 
with the property that $\mathcal F$ is strictly decreasing along its non-constant flow lines, and whose singularities outside $\Lambda_*$ are OGCs. 
The flow of the pseudo-gradient field will be denoted by $\eta_*$, and introduced formally in Proposition~\ref{prop:etaglobale}.
A suitable choice of topological invariant for subsets of $\widetilde{\mathfrak M}$ (relative category, see Definition~\ref{def:10.1}) and a modified minimax argument will allow to detect those singularities of the pseudo-gradient that are true orthogonal geodesic chords. As usual, the minimax argument requires a Palais--Smale property of the flow.\smallskip

The minimax setup requires that one considers paths satisfying an upper bound on their length. Such upper bound is denoted by $M_0$, and it appears for the first time in \eqref{eq:nonsat}. The constant $M_0$ is required to satisfy two distinct properties, see Theorem~\ref{thm:generale}. First, it must be grater than or equal to the length of every path in $\mathfrak C$, which guarantees that one has enough critical values defined by the minimax method, see \eqref{eq:10.3}. Second, it must ensure that the technical assumption \eqref{eq:nonsat} of Theorem \ref{thm:generale} is satisfied. The flow will then be defined on the $M_0^2$-sublevel of the action functional $\mathcal F$.\smallskip

Next, we determine a positive lower bound for the length/energy of $\mathcal F$-critical paths, which guarantees nontriviality (constant paths lying on the boundary of $\Omega$ are trivial $\mathcal F$-critical paths). The choice of such a lower bound is related to the geometry of 
$\overline\Omega$, more specifically, to the measure of its concavity. The proximity to $\partial\Omega$ is 
measured in terms of a \emph{regularized signed distance function} $\phi\colon\overline\Omega\to\left]-\infty,0\right]$, introduced in Section~\ref{sub:phi}. The strong concavity of $\overline\Omega$ is expressed in terms of the Hessian of $\phi$ which is assumed negative definite in the directions tangent to $\partial\Omega$. Small neighborhoods of $\partial\Omega$ in $\Omega$ are described as superlevels $\phi^{-1}\big(\left]-\delta,0\right]\big)$, for small $\delta>0$. The strong concavity assumption implies the existence of $\delta_0>0$ such that every non-trivial geodesic in $\overline\Omega$ with endpoints in $\partial\Omega$ must contain some point in $\phi^{-1}\big(\left]-\infty,-\delta_0\right]\big)$, see Remark~\ref{rem:1.4bis}. This simple observation implies that, if $K_0$ denotes the maximum of $\Vert\nabla\phi\Vert$ (formula \eqref{eq:ko}), then there cannot be nontrivial OGCs $\gamma$ in $\overline\Omega$ satisfying $\mathcal F(\gamma)\le\frac{\delta_0^2}{K_0^2}$. This gives a positive lower bound for the energy of nontrivial $\mathcal F$-critical paths. \smallskip

The pseudo-gradient dynamics on the space $\mathfrak M$ or, more precisely, on the sublevel $\mathcal F^{-1}\big([0,M_0^2]\big)$, is obtained by integrating a certain vector field on $\mathfrak M$, whose construction is one of the main technical parts of the paper. Recall that, for $x\in\mathfrak M$, the tangent space $T_x\mathfrak M$ consists of vector fields along $x$, satisfying suitable regularity assumptions, and that are tangent to $\partial\Omega$ at the endpoints of $x$. Around generic paths in $\mathfrak M$ that do not touch $\partial\Omega$ except at the endpoints, such vector field is simply the gradient field of $\mathcal F$ with respect to the natural Hilbert manifold structure of $\mathfrak M$, see Section~\ref{sub:hilbstruct}. However, the gradient field has to be modified around paths that admit interior points \emph{close} to $\partial\Omega$. Note that the lack of convexity implies that the standard gradient flow for $\mathcal F$ cannot be future complete, as it may carry paths that initially stay away from $\partial\Omega$ to paths that are somewhere tangent to $\partial\Omega$, for which the gradient flow is no longer defined. \smallskip

For the construction of the pseudo-gradient field, a typical procedure adopted here is to give a \emph{pointwise definition}, by assigning  its value along specific curves $x\in\mathfrak M$ (this means, giving a vector field \emph{along} $x$), then extended locally. The local constructions are finally merged to produce a globally defined vector field using partitions of unity, as in \cite{PalaisLS}. Given $x\in\mathfrak M$, the closed convex cone $\mathcal V^-(x)$ (see \eqref{eq:V-sigma}) is the set of $V\in T_x\mathfrak M$ that point \emph{inside} $\Omega$ at those points of $x$ that lie on $\partial\Omega$. A path $x\in\mathfrak M$ is defined to be $\mathcal V^-$-critical if the derivative of $\mathcal F$ at $x$ in the direction of every $V\in\mathcal V^-(x)$ is nonnegative. As discussed above, not every $\mathcal V^-$-critical points are OGCs.
\smallskip

The technical condition  \eqref{eq:nonsat} of Theorem \ref{thm:generale} requires the existence of some point $x_0\in\Omega$ and a suitable ball $B_{\rho_0}(x_0)$ centered at $x_0$ (see \eqref{eq:Brho}) which is not run across by what we call \emph{special geodesics}, see details in the lines preceding formula \eqref{eq:GL}. The closure of such ball has positive distance from the boundary of $\Omega$, and the constant $\delta_1$ (less than or equal to $\delta_0$), defined in \eqref{eq:delta1}, determines the superlevel of $\phi$ having empty intersection with the closure of   $B_{\rho_0}(x_0)$. 
By Corollary~\ref{smallstrip}, this implies that every $x\in
\mathcal F^{-1}\big([0,\delta_1^2/K_0^2]\big)$ lies in $\phi^{-1}\big(\left[-\delta_1,0\right]\big)$, and ultimately, by Proposition~\ref{prop:sublevels}, that  $\mathcal F^{-1}\big([0,\delta_1^2/K_0^2]\big)\bigcup\Lambda_*$ can be retracted to a set consisting of paths lying on the boundary of $\Omega$.

\textbf{Acknowledgments}. The authors gratefully acknowledge the important contribution given by Dario Corona and Isabel Beach, who raised important questions on previous versions of the paper through many discussions.

\section{Preliminaries: notations, terminology and some basic facts}
\label{sec:preliminaries}
\subsection{Riemannian preliminaries}\label{sub:RiemPrel}
Let us assume that $M$ is an $N$--dimensional differentiable manifold of class $C^3$, and that $M$ is endowed with a Riemannian metric tensor $g$ which is of class $C^2$
(this regularity guarantees the uniqueness of the solution for the Cauchy problem of geodesics).

\begin{rem}\label{rem:distictep}
Observe that if $\gamma\colon[0,1]\to\overline\Omega$ is a non-constant orthogonal geodesic chord, then $\gamma(0) \neq \gamma(1)$, by the uniqueness of the solution for the geodesic Cauchy problem. Indeed if  $\gamma(0)=\gamma(1)$ then $\gamma(t)=\gamma(1-t)$ for all $t$, hence $\dot \gamma(\frac12)=0$.
\end{rem}
The symbol $\nabla$ will denote the covariant derivative of the Levi-Civita connection of $g$, as well as the gradient differential operator with respect to $g$ on $M$. The Hessian $\mathrm H^f(q)$ of a smooth map $f\colon M\to\mathds R$ at
a point $q\in M$ is the symmetric bilinear form $\mathrm
H^f(q)(v,w)=g\big((\nabla_v\nabla f)(q),w\big)$ for all $v,w\in
T_qM$; equivalently, $\mathrm H^f(q)(v,v)=\frac{\mathrm
d^2}{\mathrm ds^2}\big\vert_{s=0} f(\gamma(s))$, where
$\gamma\colon\left]-\varepsilon,\varepsilon\right[\to M$ is the unique affinely parameterized
geodesic in $M$ with $\gamma(0)=q$ and $\dot\gamma(0)=v$. We will denote by $\Dds$ the covariant
derivative along a curve, in such a way that $\Dds\dot x=0$ is the
equation of the geodesics. A basic reference on the background material for Riemannian geometry is \cite{docarmo}.

\subsection{The distance from the boundary and the function \texorpdfstring{$\phi$}{phi}}\label{sub:phi}
Consider the distance function $\overline\Omega\ni x\mapsto d(x,\partial \Omega)$, where $d$ is the distance induced by $g$. This map is of class $C^2$ \emph{near} $\partial\Omega$. Namely, $d(x,\partial\Omega)=\Vert(\exp^\perp)^{-1}x\Vert$, where $\exp^\perp$ is the normal exponential function on $\partial\Omega$ and $\Vert \cdot \Vert$ is the norm induced by $g$.
\smallskip

Let $\delta_*>0$ be such that $x\mapsto d(x,\partial\Omega)$ is of class $C^2$ in $\big\{x\in\overline\Omega:d(x,\partial\Omega) \leq \delta_*\big\}$. 
In the rest of the paper we shall denote by $\phi$ a fixed function of class $C^2$ on $\overline \Omega$ such that
\begin{equation}\label{eq:phi}
\phi(x)=-d(x,\partial \Omega) \text{ if }d(x,\partial \Omega) \leq \delta_* \text{ and }
\phi(x) < -\delta_* \text{ if }d(x,\partial\Omega) > \delta_*.
\end{equation}
By this choice 
\begin{equation}\label{eq:gradientnorm}
\Vert \nabla \phi(x) \Vert=1, \text{ for any }x \in \phi^{-1}\big([0,\delta_*]\big).
\end{equation}

\subsection{Strong concavity}
The multiplicity result of Bos (\cite{bos}) is  proved assuming $\overline\Omega$ to have a smooth boundary and to be   convex and homeomorphic to the $N$--dimensional disk. In this case, the convexity of $\overline \Omega$ means that the Hessian of the map  $\phi$   is positive semi-definite on $\partial \Omega$ along the tangent directions to $\partial \Omega$:
\[
\mathrm H^\phi(x)(v,v) \geq 0
\text{ for any }x\in\partial\Omega,\ v\in
T_x\partial\Omega.
\]
As mentioned above, a counterxample to the existence of OGCs in nonconvex Riemannian discs is given in \cite{bos}.
In this paper we assume a strong concavity condition of $\overline \Omega$ whose definition can be again given in terms of the map $\phi$.
\begin{defin}\label{def:stronglyconcave} The domain $\overline\Omega$ is called {\em strongly concave} if 
\[
\mathrm H^\phi(x)(v,v) < 0
\text{ for any }x\in\partial\Omega,\ v\in
T_x\partial\Omega, v\not=0.
\]
\end{defin}
\noindent Note that if $\overline\Omega$ is  strongly concave, then geodesics
starting on $\partial\Omega$ tangentially to $\partial\Omega$
locally move \emph{inside} $\Omega$.

\begin{rem}\label{rem:1.4}
Strong concavity is evidently a {\em $C^2$-open condition}.
Then, by compactness, there exists $\delta_0 \in\left ]0,\delta_*\right]$
such that $H^\phi(q)[v,v]<0$ for any $q$ in $\phi^{-1}\big([-\delta_0,0]\big)$
and any $v \not=0$ such that $g(\nabla\phi(q),v)=0$.
\end{rem}

\begin{rem}\label{rem:1.4bis}
The strong concavity condition gives us the following property, that will be systematically used throughout the paper. Let $\delta_0$ be as in Remark~\ref{rem:1.4}; then:
\begin{equation}\label{eq:1.1bis}
\begin{matrix}
\text{for any non-constant geodesic $\gamma\colon[a,b]\to\overline\Omega$ with $\phi\big(\gamma(a)\big)=\phi\big(\gamma(b)\big)=-\delta \in\left]-\delta_0,0\right]$}\\
\text{and $\phi\big(\gamma(s)\big)<-\delta\,\forall s\in\left]a,b\right[$, there exists $\overline s\in\left]a,b\right[$ such that $\phi\big(\gamma(\overline s)\big)<-\delta_0$.}
\end{matrix}
\end{equation}
Such property is proved easily by a contradiction argument, looking at the minimum point of the map $s\mapsto\phi\big(\gamma(s)\big)$.
\end{rem}

\begin{rem}\label{rem:totgeo}
Using the fact that $\nabla\phi\ne 0$  in $\phi^{-1}\big([-\delta_0,0]\big)$, one obtains easily that  an open neighborhood of $\phi^{-1}\big([-\delta_0,0]\big)$ is diffeomorphic to the product $I\times\partial\Omega$, where $I\subset\mathds R$ is a neighborhood of $[-\delta_0,0]$, and with $\{-\sigma\}\times\partial\Omega$ corresponding to $\phi^{-1}(-\sigma)$ for all $\sigma\in[-\delta_0,0]$. Choosing a product metric on $I\times\partial\Omega$, one obtains a metric $\overline g$ (defined on an open neighborhood of $\phi^{-1}\big([-\delta_0,0]\big)$, and then extended to $M$) such that $\phi^{-1}(\sigma)$ is $\overline g$-totally geodesic for all $\sigma\in[-\delta_0,0]$. Denoting by $\overline\exp$ the exponential map of $\overline g$, the totally geodesic property means that there exists a sufficiently small neighborhood $\mathcal N$ of the zero section of $TM$ such that,  setting $\mathcal N_p=\mathcal N \cap T_pM$ for all $p \in  \phi^{-1}(-\sigma)$, it is:
\begin{equation}\label{eq:TGsigma}
\mathcal N_p \cap (\overline\exp_p)^{-1}(\phi^{-1}(-\sigma)) \subset T_p(\phi^{-1}(-\sigma)) \text { for all }p \in \phi^{-1}(-\sigma).
\end{equation}
This construction is made using a partition of the unity argument; more precisely, the metric $\overline g$ is obtained using the Euclidean structure on local charts having the function $\phi$ as last coordinate. Using the smoothness of the Euclidean exponential map, and the $C^3$-regularity of $M$, the corresponding exponential map $\overline\exp$ is also of class $C^3$. 
\end{rem}
\subsection{Brake orbits and OGCs}\label{sub:BOOGC}
We will give here a very short account of a geometric approach to periodic solutions of \eqref{eq:lagrsystem}, 
and at the end of Section~\ref{sub:varframe} we will  show how to obtain a proof of Seifert conjecture using the multiplicity of OGCs.

Let $(\widehat M,\widehat g)$ be a Riemannian $N$--dimensional manifold representing the configurations space \eqref{eq:lagrsystem}.
By the classical Maupertuis principle, solutions of \eqref{eq:lagrsystem} having energy $E$ are, up to a parameterization, geodesics in the conformal metric $g_E$, the \emph{Jacobi metric}:
\begin{equation}\label{eq:maupertuismetric}
g_E=\big(E-V(p)\big)\cdot\widehat g,
\end{equation}
defined in the open $E$-sublevel 
$M^E=V^{-1}\big(\left]-\infty,E\right[\big)$ of $V$. Observe that, in fact, $g_E=0$  on the boundary $\partial M^E=V^{-1}(E)$. 

Thus, brake orbits correspond to $g_E$-geodesics
in $M^E$ with endpoints in $\partial M^E$, or, more precisely, to $g_E$-geodesics
$\gamma\colon\left]0,T\right[\to V^{-1}\big(\left]-\infty,E\right[\big)$, with $\lim\limits_{t\to0^+}\gamma(t)$ and $\lim\limits_{t\to T^-}\gamma(t)$ in $\partial M^E$.

For such a degenerate situation, it is proved in \cite{GGP1,cag} that, if $E$ is a regular value of the function $V$ (which implies in particular that $\partial M^E=V^{-1}(E)$ is a smooth hypersurface of $\widehat M$), then $g_E$ defines a distance-to-the-boundary function
\[
\mathrm{dist}_E\colon M^E\longrightarrow\left[0,+\infty\right[
\] which is of class $C^2$ in  $M^E$, near $\partial M^E$, and which extends continuously to  $0$ on the boundary $\partial M^E$. 
As in the nonsingular case, the distance from the boundary function is given in terms of an infimum of length of curves: 
\begin{multline*}
\mathrm{dist}_E(Q,V^{-1}(E))=\\
\inf\Big\{
\int_0^1\sqrt{(E-V(x))g(\dot x, \dot x)}\,\mathrm ds: 
x \in C^1\big([0,1],V^{-1}(]-\infty,E]\big),\\
 x(0) \in V^{-1}(E),\ V(x(s)) < E \text{ for all }s \in\left]0,1\right],\ x(1)=Q\Big\}.
\end{multline*}

\begin{prop}\label{thm:collective}
There exists $\widehat \delta$ such that for any $\delta \in\left ]0,\widehat\delta\right]$, any OGC in the
Riemannian manifold with boundary ${M^E_\delta}=\mathrm{dist}_E^{-1}\big(\left[\delta,+\infty\right[\big)\bigcap V^{-1}\big(\left]-\infty,E\right[\big)$ endowed with the metric $g_E$ (which is now non-singular) can be extended uniquely to a $g_E$-geodesic $\gamma$ in the potential well $\overline{M^E}=V^{-1}\big(\left]-\infty,E\right]\big)$ with endpoints in $\partial M^E$, and $M_\delta^E$ is homeomorphic to  $\overline{M^E}$. 

Moreover, (when $M^E$ is endowed with the metric $g_E$) $M_\delta^E$ is strongly concave.  More precisely,
the Hessian of the distance-to-the-boundary function $\mathrm{dist}_E$ satisfies:
\begin{equation}\label{eq:concavjacobi}
\begin{aligned}
\mathrm H^{\mathrm{dist}_E}(x)(v,v) > 0&
\text{ for every }x\in M^E\ \text{such that } 0<\mathrm{dist}_E(x) \leq \widehat\delta,\\
 \text{ and every }& \text{ $v\ne0$ such that } g_E\big(\nabla \mathrm{dist}_E(x),v\big)=0.
\end{aligned}
\end{equation}
\end{prop}
\begin{proof}
See \cite[Proposition~5.8 and Theorem~5.9]{GGP1}.
\end{proof}
In order to prove Seifert's conjecture by multiplicity of OGCs, we will need also the following result, whose proof is obtained from Maupertuis' Principle and from the property that, on the boundary of the potential well, the gradient of the potential function is nowhere vanishing.
\begin{lem}\label{lem:jconv}
Assume that $\gamma_n\colon[0,1]\to V^{-1}\big(\left]-\infty,E\right[\big)$ is a sequence of Jacobi geodesics and $L>0$ such that 
\begin{itemize}
\item[(a)]  $\gamma_n(0) \to P \in V^{-1}(E)$ as $n\to\infty$;\smallskip

\item[(b)] $\dist_E\big(\gamma_n(1)\big)=\dist_E\big(\gamma_n(0)\big)$, and  $\dist_E\big(\gamma_n(s)\big)>\dist_E\big(\gamma_n(0)\big)$ for all $s \in\left]0,1\right[$;\smallskip

\item[(c)] $0<\int_0^1 g_E(\dot \gamma_n,\dot \gamma_n)\,\mathrm ds\leq L^2$, for all $n$.
\end{itemize}
Then there exists a brake orbit $q$ starting from $P$, such that 
\begin{equation}\label{eq:limitbo}
\lim_{n\to\infty}\ \sup_{s\in[0,1]}\widehat d\big(\gamma_n(s),q(\mathds{R})\big)=0,
\end{equation}
where $\widehat d$ is the distance function relative to the metric $\widehat g$ of the configuration space.
\end{lem}
\begin{proof}
First note that, by the strong concavity condition in  \eqref{eq:concavjacobi}, there exists $n_0$ and $\varepsilon_* > 0$ such that
\begin{equation}\label{eq:unifdistfromboundary}
\max\Big\{E-V\big(\gamma_n(s)\big): s \in [0,1]\Big\} \geq \varepsilon_*,\  \text{ for every }n \geq n_0,
\end{equation}
see 
Remark~\ref{rem:1.4bis}.
Set $c_n=\int_0^1g_E\big(\dot \gamma_n(s),\dot \gamma_n(s)\big)\,\mathrm ds$; by assumption (c), $c_n\le L^2$ for all $n$. 
Set also
\begin{equation}\label{eq:pr2}
t_n(s)=
\frac{1}{\sqrt{2}}\int_0^s\frac{\sqrt{c_n}}{E-V(\gamma_n(\tau))}\,\mathrm d\tau.
\end{equation}
Denote by  $\sigma_{n}(t)$ the inverse of $t_{n}$, and consider $q_n(t) = \gamma_n\big(\sigma_{n}(t)\big)$. Since $c_n > 0$, a straightforward computation shows that $q_n$ is a solution of the ODE:
\begin{equation}\label{eq:cauchy}
\Ddt\dot q_n(t) + \nabla V\big(q_n(t)\big)= 0,
\end{equation}
where $\Ddt$ and $\nabla $ denote the covariant derivative and  the gradient relatively to the metric $\widehat g$,
while 
\[\tfrac12\widehat g(\dot q_n,\dot q_n)+V\big(q_n(t)\big)=E,\quad \text{ and }\quad q_n(0)=\gamma_n(0) \longrightarrow P \in V^{-1}(E).\] 
Note that the $g_E$-length of $q_n$, denoted by $L(q_n)$, coincides with that of $\gamma_n$, and therefore, by assumption (c):
\begin{equation}\label{eq:lengthqn}
\phantom{,\qquad\forall n.} L(q_n)\le L,\qquad\forall n.
\end{equation}
In order to conclude the proof, it suffices to show that $t_n(1)$ is bounded. Indeed, if this is true, the proof is immediately concluded by passing to the limit in \eqref{eq:cauchy}, because, by \eqref{eq:unifdistfromboundary}, the limit curve $q$ must then be a (noncostant) brake orbit.

To prove that $t_n(1)$ is bounded, let us set $\rho_n(t)=E-V\big(q_n(t)\big)$; we have
\begin{equation}\label{eq:formrho'n}
\dot \rho_n(t)=-\widehat g\big(\nabla V(q_n(t)\big),\dot q_n(t)\big)
\end{equation}
and:
\begin{equation}\label{eq:ddotrhon}
\ddot \rho_n(t) = -H^V(q_n(t)\big(\dot q_n(t)],\dot q_n(t)\big)+\widehat g\big(\nabla V(q_n(t)), \nabla V(q_n(t))\big).
\end{equation}
where  $H^V$ is the  Hessian of $V$ with respect to the Riemann structure $\widehat g$. Note that, from \eqref{eq:formrho'n} and the equality $\frac12\widehat g(\dot q_n,\dot q_n)=E-V(q_n)$, we obtain that $\dot\rho_n$ is uniformly bounded.\smallskip

Now, set
\[
C=\inf\big\{-H^V(x)(v,v): x \in V^{-1}\big(\left]-\infty,E\right]\big),\ v \in T_x M,\  g(v,v)=1\big\}.
\]
Fix $\nu_0>0$ and $\varepsilon_0\in\left] 0,\varepsilon_*\right[$ such that 
\[
2C\big(E-V(x)\big)+\widehat g\big(\nabla V(x), \nabla V(x)\big) \ge \nu_0,\quad \text{for every $x$ such that $0\leq E-V(x) \leq \varepsilon_0$.}
\]
From \eqref{eq:ddotrhon}, we obtain:
\begin{multline*}
\ddot\rho_n(t)\ge C\cdot\widehat g(\dot q_n,\dot q_n)+\widehat g\big(\nabla V(q_n),\nabla V(q_n)\big)\\=2C\big(E-V(q_n)\big)+\widehat g\big(\nabla V(q_n),\nabla V(q_n)\big),
\end{multline*}
and therefore:
\begin{equation}\label{eq:dersecpos}
 \ddot\rho_n(t)\ge\nu_0 > 0,\quad \text{when}\ q_n(t)\  
 \text{belongs to set}\ \mathcal A_{\varepsilon_0}:=
 \big\{x:E-V(x)\leq \varepsilon_0\big\}.
 \end{equation}
Recalling that $\dot\rho_n$ is uniformly bounded, such lower bound on $\ddot\rho_n$ implies that 
\begin{equation}\label{eq:boundlengtintervals}
\begin{aligned}
\text{ there exists a uniform} & \text{ lower bound on the length}\\
\text{  of any interval of time $t$} &\text{ for which $q_n(t)$ belongs to $\mathcal A_{\varepsilon_0}$}. 
\end{aligned}  
\end{equation}

Let us call \emph{of type 1} every interval $[a,b]\subset\big[0,t_n(1)\big]$ such that $q_n\big([a,b]\big)\subset\mathcal A_{\varepsilon_0}$ (this notion clearly depends on $n$). Similarly, let us call \emph{of type 2} every interval $[a,b]\subset\big[0,t_n(1)\big]$ such that $q_n\big([a,b]\big)\subset\mathcal B_{\frac12\varepsilon_0}:=\big\{x:E-V(x)\geq \frac12\varepsilon_0\big\}$, which is maximal with respect to this property. Clearly, $\big[0,t_n(1)\big]$ can be written (not uniquely!) as union of intervals that are either of type 1 or of type 2. 

From \eqref{eq:dersecpos} it follows that, if $[a,b]$ is an interval of type $2$, the maximum on $[a,b]$ of the map $s\mapsto E-V\big(q_n(s)\big)$ is attained at some instant $\bar s$ such that $q_n(\bar s)\not\in\mathcal A_{\varepsilon_0}$. Thus, an interval of type $2$ cannot be also of type $1$. It also follows that, given an interval $[a,b]$ of type $2$, the $g_E$-length of
$q_n\big\vert_{[a,b]}$ is at least twice the \emph{Jacobi} distance between the level hypersurfaces $\Sigma_{\varepsilon_0}$ and $\Sigma_{\frac12\varepsilon_0}$, where $\Sigma_a=\big\{x:E-V(x)=a\big\}$.

This shows that there exists a uniform upper bound on the number of distinct intervals of type $2$. For each $n$, the complement in $\big[0,t_n(1)\big]$ of the union of all intervals of type 2 must then consist of a (uniformly bounded) finite number of intervals, which are necessarily of type 1, and therefore 
by \eqref{eq:boundlengtintervals}
they have uniformly bounded length. In conclusion, $t_n(1)$ is bounded.
\end{proof}
The result of Lemma~\ref{lem:jconv} will be employed in Proposition~\ref{prop:SGL} to establish that, if the number of brake orbits is finite, then the Jacobi metric satisfies assumption \eqref{eq:nonsat} of Theorem \ref{thm:generale}.
\section{The functional framework}\label{sub:varframe}
\subsection{Hilbert structure and distance function}\label{sub:hilbstruct}

Let $(M,g)$ be a Riemannian manifold ($M$ of class $C^3$ and $g$ of class $C^2$), and let $\Omega\subset M$ be an open subset of $M$ whose boundary $\partial \Omega$ is a hypersurface of class $C^2$.  
For any $[a,b] \subset [0,1]$, $H^1\big([a,b],M\big)$ will denote the Sobolev
space of all absolutely continuous curves $x\colon[a,b]\to M$
whose weak derivative is square integrable in any local chart of the manifold $M$.

For $S\subset M$  define:
\begin{equation*}
H^{1}\big([a,b],S\big)=\big\{x\in H^{1}\big([a,b],M\big):x(s)\in S
\text{ for all }s\in [a,b]\big\},
\end{equation*}

It will be useful to have a background linear structure, and for this we appeal to the classical Whitney Embedding Theorem (\cite{wh}). Thus, we will assume that
$M$ is embedded\footnote{%
 Among other things, considering $M$ embedded in $\mathds R^{m}$ will give us a notion of  \emph{weak} $H^1$-convergence of sequences of curves in $M$.} in $\mathds{R}^m$, with $m=2N$. 
Once such an embedding has been chosen, we can define a distance $\dist_*$ on $H^{1}\big([0,1],M\big)$ setting:   
 \begin{multline}\label{eq:distanza}
\dist_{*}(x_2,x_1)=\Big(\int_0^1\big \Vert \dot x_2(s)- \dot x_1(s) \big\Vert_m^2 ds\Big)^{\frac12}\\+ \max\big\{\big\Vert x_2(0)-x_1(0) \big\Vert_m,\  \big\Vert x_2(1)-x_1(1)\big\Vert_m\big\}.
\end{multline}
where $\Vert \cdot \Vert_m$ is the Euclidean norm in $\mathds R^m$.
Moreover, in $T_{x}H^{1}\big([0,1],M\big)$  we consider the norm 
\begin{equation}\label{eq:normaV}
\Vert V \Vert_{*} = \max\big\{\Vert V(0)\Vert_m,\ \Vert V(1)\Vert_m\big\}+
\Big(\int_0^1 \big\Vert V'  \big\Vert_m^2 ds\Big)^{\frac12},
\end{equation}
where $V'$ is the usual derivative of $V$ as a map from $[0,1]$ to $\mathds{R}^m$.

Using the exponential map in Remark~\ref{rem:totgeo}, one proves that $H^{1}\big([a,b],M\big)$ is an infinite dimensional Hilbert  manifold of class $C^2$ (more precisely, a $C^2$--submanifold of $H^1\big([a,b],\mathds{R}^m\big)$), modeled on the Hilbert space $H^{1}\big([a,b],\mathds{R}^N\big)$. For $x\in H^1\big([a,b],M\big)$, the tangent space $T_xH^1\big([0,1],M\big)$ is given by
\[
T_{x}H^{1}\big([a,b],M\big)=\big\{\xi \in H^{1}\big([a,b],TM\big): \xi(s) \in T_{x(s)}M \text{ for all }s \in [a,b]\big\},
\]
where $TM$ denotes the tangent bundle of $M$. 

\subsection{The admissible paths and the energy functional}
 Let us consider the following set of paths:
\begin{equation}\label{eq:defM}
\mathfrak M=\Big\{x\in
H^{1}\big([0,1],\overline\Omega \big):
x(0)\in \partial \Omega, x(1) \in \partial\Omega \Big\}.
\end{equation}
We will use the \emph{geodesic action functional}
$\mathcal F$ on $\mathfrak M$, defined by:
\begin{equation}\label{eq:enfunct}
\mathcal F(x)=\int_0^1g(\dot x, \dot x)\,\mathrm ds.
\end{equation}
The differential $\mathrm d\mathcal F$ in $H^1([0,1],M)$ is easily computed as:
\begin{equation}\label{eq:dermathcalF}
\mathrm d\mathcal F(x)[V]=2\int_0^1g\big(\dot x,\Dds V\big)\,\mathrm ds,
\end{equation}
for all $x\in\mathfrak M$ and all $V\in T_xH^1([0,1],M)$.

Define the  constant:
\begin{equation}\label{eq:ko}
K_0=\sup_{x \in \overline\Omega}\Vert \nabla \phi(x)\Vert.
\end{equation}

The following result will be systematically used in the rest of the paper:
\begin{lem}\label{thm:lem2.3}
Let $x\in\mathfrak M$ and let $[a,b]\subset[0,1]$. Then
\begin{equation}\label{eq:aggiuntalemma2.1}
\big|\phi(x(b))-\phi(x(a))\big|\leq  K_0\left((b-a) \int_a^bg(\dot x,\dot x) \,
\mathrm d\sigma\right)^{\frac12}.
\end{equation}
\end{lem}
\begin{proof}
Since   $\Vert \nabla \phi(x) \Vert \leq K_0$ for any $x \in \overline\Omega$, for any $s\in [a,b]$ we have 
\begin{multline*}
\big\vert\phi\big(x(s)\big)-\phi\big(x(a)\big)\big\vert\le\int_a^{s}\big\vert g\big(\nabla\phi\big(x(\sigma)\big),\dot x(\sigma)\big)\big\vert\,
\mathrm d\sigma \le \\
\le K_0\int_a^sg(\dot x,\dot x)^{\frac12}\,  \mathrm
d\sigma
\le K_0\sqrt{s-a} \left(\int_a^sg(\dot x,\dot x) \,\mathrm
d\sigma\right)^{\frac12},
\end{multline*}
from which inequality  \eqref{eq:aggiuntalemma2.1} follows.
\end{proof}

\begin{cor}\label{smallstrip} Let $\delta > 0$. 
Let $x \in \mathfrak M$ be such that $\phi\big(x(a)\big)=0$ for some $a\in\left[0,1\right[$, and assume that for some $b\in\left]a,1\right]$:
\[
\int_a^bg(\dot x, \dot x)\,\mathrm d\sigma \leq \frac{\delta^2}{K_0^2}.
\] 
Then $\phi\big(x(s)\big)\geq -\delta$ for all $s \in [a,b]$.\qed
\end{cor}
\subsection{\texorpdfstring{$\mathbf{\mathds Z_2}$}{Z2}-equivariance}
Consider the diffeomorphism $\Rcal\colon\mathfrak M \to \mathfrak M$:
\begin{equation}\label{eq:2.4}
\Rcal x(t)=x(1-t).
\end{equation}
We say that $\Ncal\subset \mathfrak M$ is
\emph{$\Rcal$--invariant} if $\Rcal(\Ncal)=\Ncal$; note that
$\mathfrak M$ is $\Rcal$-invariant.

\begin{rem}\label{rem:noROGC}
Note that if $\gamma\colon [0,1]\to \overline\Omega$ is a $C^1$--curve such that $\dot \gamma(s) \not=0$ for any $s \in [0,1]$, then\footnote{
$\gamma\ne\mathcal R\gamma$ as a point of $\mathfrak M$. On the other hand, if $\gamma$ is an OCG, then $\gamma$ and $\mathcal R\gamma$ are \emph{not} geometrically distinct as OGCs.} $R\gamma\not= \gamma$.  Indeed if by contradiction $R\gamma = \gamma$, then $\gamma(1-t)=\gamma(t)$ for any $t$, from which we deduce $\dot \gamma(\frac12)=0$.
\end{rem}

The following Lemma will be used to define an $\Rcal$--invariant subset $\mathfrak C$ of ${\mathfrak M}$
which carries the main topological properties of ${\mathfrak M}$.
 
\begin{lem}\label{thm:corde} There exists  a continuous map $\gamma\colon\partial\Omega\times\partial\Omega\to H^1\big([0,1],\overline\Omega\big)$ such that
\begin{enumerate}
\item\label{corde1} $\gamma_{A,B}(0)=A,\,\,\gamma(A,B)(1)=B$.
\item\label{corde3} $\gamma_{A,A}(t)=A$, for all $t\in [0,1]$.
\item\label{corde4} $\Rcal \gamma_{A,B}=\gamma_{B,A}$, namely $\gamma_{A,B}(1-t)=\gamma_{B,A}(t)$ for all $t$, and for all $A,B$.
\end{enumerate}
\end{lem}
\begin{proof}
Let $\Psi\colon \overline \Omega  \rightarrow {\mathds D}^N$ be an  homeomorphism, where ${\mathds D}^N$ is the unit disk in $\mathds{R}^N$. Define
\[
\phantom{,\quad A,B\in \overline\Omega.}\gamma_{A,B}(t) = \Psi^{-1}\big((1-t)\Psi(A)+t\Psi(B)\big),\quad A,B\in \overline\Omega.
\]
If $\Psi$ is of class $C^1$, the above formula gives the desired map $\gamma$.
In general, if $\overline\Omega$ is only homeomorphic (and not diffeomorphic) to the disk $\mathds D^N$,  and the map $\Psi$ is only continuous, the above definition produces curves may not have $H^1$-regularity. 
However, starting from the above map $\gamma$ it is not difficult to obtain a map  taking values in $H^1\big([0,1],\mathds{R}^m\big)$ 
and satisfying \eqref{corde1}--\eqref{corde4} using 
a broken geodesic approximation argument.
\end{proof}

We will later choose a map $\gamma$ as in Lemma~\ref{thm:corde} satisfying an additional property, see Remark~\ref{thm:remgammaM0}.\smallskip

Denote by $\mathfrak C$ the image of the map $\gamma$ as above and by $\mathfrak C_0\subset\mathfrak C$ the image of the diagonal of $\partial\Omega\times\partial\Omega$:
\begin{equation} \label{eq:CC0}
\begin{aligned}
&{\mathfrak C}=\Big\{\gamma_{A,B}\,:\,A,B\in\partial\Omega\Big\}\\
&
{\mathfrak C_0}=\Big\{\gamma_{A,A}:A\in\partial\Omega \Big\}.
\end{aligned}
\end{equation}
Note that $\mathfrak C_0$ is the set of the constant curves in $\mathfrak M$.

\subsection{On the geometric technical assumption}
This section is devoted to the definition and a discussion on the technical geometric assumption mentioned in the introduction.

Recalling the definition of the functional $\mathcal F$ in \eqref{eq:enfunct}, we set:
\begin{equation}\label{eq:topconst1}
L^{(2)}(\gamma)=\sup\big\{ \mathcal F\big(\gamma(A,B)\big): A,B\in \partial \Omega,\; \gamma \text{ satisfies  \eqref{corde1}--\eqref{corde4} of Lemma \ref{thm:corde}}\big\},
\end{equation}
and 
\begin{equation}\label{eq:topconst2}
L^{(2)}(\Omega)=\inf\big\{L(\gamma)^2:\gamma \text{ satisfies  \eqref{corde1}--\eqref{corde4} of Lemma \ref{thm:corde}}\big\}.
\end{equation}
\begin{rem}\label{rem:stimaMOmega} Note that 
\[
L^{(2)}(\Omega)\geq \frac{\delta_0^2}{K_0^2}.
\]
Indeed, if  $L^{(2)}(\Omega) < {\delta_0^2}/{K_0^2}$, then there would exist a map $\gamma$ as in Lemma~\ref{thm:corde} such that $L^2(\gamma) < {\delta_0^2}/{K_0^2}$. In this case, by  Corollary~\ref{smallstrip}, the set $\mathfrak C$ could be deformed continuously onto a set consisting of curves with image in $\partial \Omega$, which is not possible.
\end{rem} 
 
Let us introduce a suitable class of geodesics in $\overline\Omega$ with endpoints in $\partial\Omega$.
For the purpose of giving a short definition of the set $\mathcal G_L$ below \eqref{eq:GL}, let us temporarily call \emph{special} a geodesic $x\colon[0,1]\to\overline\Omega$ with $x\big(\left]0,1\right[\big)\subset \Omega$ and  $x(0),x(1)\in\partial\Omega$ if it satisfies one of the two boundary conditions below:
\begin{itemize}
\item[(a)]  $x$ is tangent to $\partial\Omega$ at both endpoints;
\item[(b)]  $x$ is tangent to $\partial\Omega$ at one endpoint, and orthogonal to $\partial\Omega$ at the other endpoint.
\end{itemize}
 For $L > 0$ fixed, we set
\begin{equation}\label{eq:GL}
\mathcal G_L=\Big\{x \colon[0,1]\rightarrow \overline \Omega: x \text{ is a special geodesic with } 0 <\mathcal F(x) \leq L^2\Big\}.
\end{equation}

The main result of the paper is the following:
 \begin{teo}\label{thm:generale}
Let $\Omega$ be an open subset of $M$ with boundary $\partial \Omega$ of class $C^2$. 
Suppose that  $\overline\Omega$ is strongly concave and homeomorphic to an $N$--dimensional disk. 

Assume also that 
\begin{equation}\label{eq:nonsat}
\begin{aligned}
\text{there exists }M_0^2 > L^{(2)}(\Omega) &\text{ and }x_0 \in \Omega \text{ such that }\\
x_0 \not\in \gamma\big([0,1]\big) \text{ for all }&\gamma \in \mathcal G_{M_0^2}.
\end{aligned}
\end{equation}
Then, there are at least $N$ geometrically distinct\footnote{%
Two orthogonal geodesic chords $\gamma_1,\gamma_2\colon[0,1]\to\overline\Omega$ are geometrically distinct if $\gamma_1\big([0,1]\big)\ne\gamma_2\big([0,1]\big)$.}
orthogonal geodesic chords in $\overline\Omega$.
\end{teo}
\begin{rem}\label{thm:remgammaM0}
Note that a map $\gamma$ as in Lemma~\ref{thm:corde} can be chosen so that 
$\mathcal F\big(\gamma_{A,B}\big) < M_0^2$ for all $A,B \in \partial \Omega$, because $M_0^2>L^{(2)}(\Omega)$.
In the remainder of the paper, we will assume that $\gamma$ has been chosen so that this property is satisfied.
\end{rem}
The proof of Theorem~\ref{thm:generale} will occupy the rest of the paper, and will finalized in Section~\ref{sec:CVO}.
\subsection{From Theorem~\ref{thm:generale} to the proof of Seifert's conjecture}
Now let us see as Theorem~\ref{thm:generale} can be used to prove Seifert conjecture. For this, use the notation of 
Section~\ref{sub:BOOGC} and denote by $\Omega_\delta$ the interior of $M^E_\delta$, with $\delta \in \big]0,\smash{\widehat\delta}\big]$, and $\widehat\delta$ given in Proposition~\ref{thm:collective}. For any $\delta \in\left]0,\smash{\widehat\delta}\right]$ consider the corresponding $L^{(2)}(\Omega_\delta)$. It is not difficult to realize that
\begin{equation}\label{eq:Mbuca}
\sup\Big\{L^{(2)}(\Omega_\delta): \delta \in \big]0,\smash{\widehat\delta}\big]\Big\} < +\infty.
\end{equation}
Then from Lemma~\ref{lem:jconv}, we obtain:
\begin{prop}\label{prop:SGL} 
Suppose that the number of brake orbits is finite. Then, there exists $\delta \in\big]0,\widehat\delta\big]$ such that $\Omega_\delta$ satisfies \eqref{eq:nonsat}.
\end{prop}
\begin{proof}
Let us choose a sequence $\delta_n>0$ with $\lim\limits_{n\to\infty}\delta_n=0$.
Lemma~\ref{lem:jconv} says that if $\gamma_n$ is a sequence of geodesics with endpoints on $\partial\Omega_{\delta_n}$, that remains uniformly away from the boundary and with bounded length, then $\gamma_n$  \emph{converges} to some brake orbit in $\Omega$. Thus, if for infinitely many $n$ the set $\Omega_{\delta_n}$ did not satisfy \eqref{eq:nonsat}, then through every point $x_0\in\Omega$ there would be an OGC. This is only possible if there were infinitely many OGCs in $\overline\Omega$.
\end{proof}
We are now ready to show how to obtain a proof of Seifert's conjecture from Theorem~\ref{thm:generale} and Proposition~\ref{prop:SGL}.
\begin{proof}[Proof of Seifert's Conjecture from Theorem~\ref{thm:generale} and Proposition \ref{prop:SGL}]\hfill\break
Let us assume that the number of brake orbits of energy $E$ is finite. From Proposition~\ref{prop:SGL}, there exists $\delta \in\big]0,\widehat\delta\big]$ such that $\Omega_\delta$ satisfies \eqref{eq:nonsat}.

From Proposition~\ref{thm:collective}, $\Omega_\delta$ is a Riemannian $N$-disk with strongly concave boundary. In addition, every OGC in $\Omega_\delta$ can be uniquely extended to a reparameterized brake orbit of energy $E$. Therefore, using Theorem~\ref{thm:generale} (applied to $M=M^E$, $g=g_E$ and $\Omega=\Omega_\delta$), we have at least $N$ geometrically distinct brake orbits of energy $E$, proving Seifert's Conjecture.
\end{proof}

\section{\texorpdfstring{$\mathcal V^-\!\!$}{V-}--critical curves and \texorpdfstring{$\mathcal V^-\!\!$}{V-}--Palais-Smale sequences}\label{sec:critPS}
Given $x\in\mathfrak M$ (cf. \eqref{eq:defM}), consider  vector fields $V\in T_xH^1\big([0,1],M\big)$ satisfying: 
\begin{equation}\label{eq:Vbordo}
g\big(\nabla \phi(x(0)),V(0)\big)=g\big(\nabla \phi(x(1)),V(1)\big)=0,
\end{equation}
and
\begin{equation}\label{eq:Venter}
g\big(\nabla \phi(x(s)),V(s)\big)\leq 0 \text{ for any } s\in\left]0,1\right[\ \text{such that}\ x(s)\in\partial\Omega.
\end{equation}
We also set
\begin{equation}\label{eq:V-sigma}
\mathcal V^-(x)=\Big\{V \in T_xH^{1}\big([0,1],M\big) \text{ satisfying \eqref{eq:Vbordo} and \eqref{eq:Venter}}\Big\}
\end{equation} 
Condition \eqref{eq:Venter} says that 
$V(s)$ does not point outside  $\Omega$ when $x(s)\in\partial\Omega$, see Figure~\ref{fig:1}. 
\begin{figure}
\centering
\includegraphics[scale=.7]{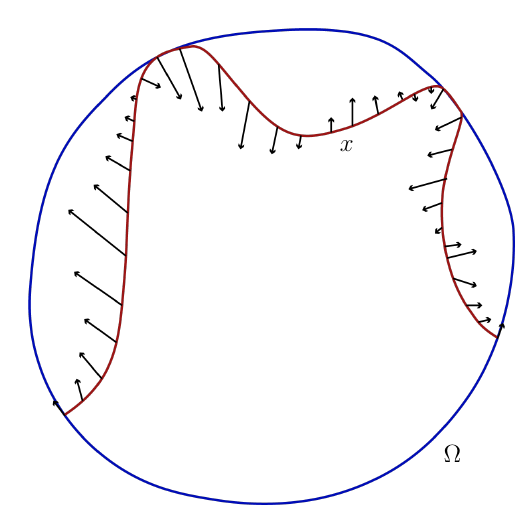}
\caption{A typical $\mathcal V^-$-field along a curve $x\in\mathfrak M$. The $\mathcal V^-$-critical points are geodesics with obstacle.}
\label{fig:1}
\end{figure}
\medskip
 
Taking inspiration from the \emph{weak slope theory}, see for instance \cite{CD,degmarz},  we then give the following:
\begin{defin}\label{def:criticalcurve}
We say that $x$ is a $\mathcal V^-$-critical curve for $\mathcal F$ on $\mathfrak M$ if
\begin{equation}\label{defcc1}
\phantom{,\quad \text{for all}\ V \in {\mathcal V}^-(x).}
\int_0^1g\big(\dot x,\Dds V\big)\,\mathrm ds\geq 0,\quad \text{for all } V \in \mathcal V^-(x).
\end{equation}
\end{defin}
Note that the set $\mathcal F^{-1}(0)$ consists entirely of minimum points in $\mathfrak M$ (the constant curves in $\partial \Omega$) which are obviously $\mathcal V^-$-critical curves. 
\subsection{Orthogonal geodesic chords with obstacle}\label{subs:OGCob}
Note that, by \eqref{eq:gradientnorm}, $\nabla \phi(p)$ is the unit exterior normal to $\overline \Omega$ at any $p \in \partial \Omega$.
In order to describe the $\mathcal V^-$-critical curves of $\mathcal F$ corresponding to positive critical levels, let us recall the following result from \cite{OT,MS}.

\begin{prop}\label{prop:descrizioneV^-}
Let $z$ be a $\mathcal V^-$--critical curve, then $\Dds\dot z$ is in $L^\infty$ (so $z$ is of class $C^1$), and $z$ is parameterized with constant speed. The portions of $z$ that lie in $\Omega$ are geodesics, while at almost all instants $s$ such that $z(s) \in \partial\Omega$, the second derivative $\Dds\dot z$ is orthogonal to $\partial \Omega$ and it points outside $\Omega$. More precisely, $z$ satisfies the equation:
\begin{equation}\label{eq:8.1bis}
\Dds\dot z= \underbrace{-H^{\phi}(z)(\dot z],\dot z)}_{\geq 0}
\nabla\phi\big(z(s)\big) \text{ for a.e. $s$ such that }z(s)\in\partial\Omega,
\end{equation} 
where $H^\phi$ is the Hessian of $\phi$ relatively to the metric $g$. Moreover    
\begin{equation}\label{lem:orth}
\dot z(0) \text{ and }\dot z(1) \text{ are orthogonal to }\partial \Omega.
\end{equation}
\end{prop}

\begin{proof}
See \cite{OT}, Proposition 3.2 and Lemmas 3.3 and 3.4.
\end{proof}

\begin{rem} Note that, in the strongly concave case, if $z$ is a nonconstant $\mathcal V^-$-curve,  the second derivative of $z$ points stricly outside $\Omega$ at those points where $z$ touches $\partial \Omega$.
\end{rem}

\subsection{Geometry of the geodesics with obstacle}\label{sub:obstpbst}
In this subsection we will describe some properties of the set of geodesics with obstacle, needed for the proof of Theorem~\ref{thm:generale}. In particular, in Remark~\ref{lem:nox0} below we will clarify the role of assumption \eqref{eq:nonsat}.  

\begin{rem}\label{rem:tangenza}
If $x$ is a  $\mathcal V^-$--critical curve, since it is of class $C^1$, if $x(s) \in \partial\Omega$ then $\dot x(s) \in T_{x(s)}\partial\Omega$.
\end{rem}

\begin{rem}\label{rem:concavconseg}
By the strong concavity and Lemma \ref{thm:lem2.3}, if $x$ is a   $\mathcal V^-$--critical curve, then the \emph{contact set} 
\begin{equation}\label{eq:contatto}
C_x=\big\{s:[0,1]: x(s) \in \partial\Omega\big\}
\end{equation}
is given by a finite number of closed intervals (possibly consisting of isolated instants). 
Such a number is uniformly bounded  on the set of
$\mathcal V^-$-critical curves $x$ satisfying $\mathcal F(x) \leq M_0^2$. 
Namely, 
from Remark~\ref{rem:1.4bis}, every geodesic in $\Omega$ starting at $\partial\Omega$ must cross the hypersurface $\phi^{-1}(-\delta_0)$ before arriving on $\partial\Omega$ again. Moreover, from Lemma~\ref{thm:lem2.3}, if $x\in\mathfrak M$, with $\mathcal F(x)\le M_0^2$, and $[a,b]\subset[0,1]$ is such that $x(a)\in\phi^{-1}(0)$, $x(b)\in\phi^{-1}(-\delta_0)$, then $b-a\ge\delta_0^2/K_0^2M_0^2$.
\end{rem}

Now  set
\begin{equation}\label{eq:Z-sigma}
\begin{aligned}
Z^-=\Big\{x \in \mathfrak M:\ &   \ x \text{  is a $\mathcal V^{-}$--critical curve with }0 <\mathcal F(x) \leq M_0^2 \\
&\text{ and }\exists\ s\in\left]0,1\right[ \text{ such that }x(s) \in \partial \Omega\Big\}.
\end{aligned}
\end{equation}

\begin{rem} 
Elements of $Z^-$ are either orthogonal geodesic chords with obstacle, or WOGC's.  If $z$ is a $\mathcal V^-$--critical curve and $z \not\in Z^-$, then $z$ is an OGC.
\end{rem}

\begin{prop}\label{rem:Zsigma-compact}
The set $Z^-$ is compact.
\end{prop}
\begin{proof}
 
Let $x_n$ be a sequence in $Z^-$. By Proposition \ref{prop:descrizioneV^-}, $g(\dot x_n,\dot x_n)$ is constant and therefore pointwise bounded by $M_0^2$. Then, by \eqref{eq:8.1bis}, the second derivative $\frac{\mathrm D}{\mathrm dt}\dot x_n$ is uniformly bounded, so, up to subsequences, $x_n$ is $C^1$-convergent to some curve $x$. By Remark \ref{rem:concavconseg}, the contact sets $C_{x_n}$ consist of a finite numbers of intervals, bounded independently of $n$. Taking the limit in  \eqref{eq:8.1bis} gives that $x \in Z^-$. 
In order to see this, it suffices to show that $x$ touches $\partial\Omega$ at some instant $s\in\left]0,1\right[$. Namely, for any $n$ there exists 
$s_n \in\left]0,1\right[$ such that $\phi\big(x_n(s_n)\big)=0$. Now, up to taking a subsequence, $s_n \to s$ as $n\to\infty$, and clearly, $\phi\big((x(s)\big)=0$. Moreover by strong concavity $x$ cannot be constant, 	while $\dot x(0)$ and $\dot x(1)$ are orthogonal to $\partial \Omega$. This implies that $s \in\left]0,1\right[$, which concludes the proof.         
\end{proof}

\begin{rem}
When the set $Z^-$ is empty, then a proof of the Deformation Lemmas for our minimax setup (Section~\ref{sec:CVO}) can be obtained by classical arguments. The interesting, and more involved,  case is when $Z^-\not=\emptyset$. Under these circumstances, our proof of the deformation lemmas requires the construction of a certain invariant set $\Lambda_*$, see Section  \ref{sec:V+criticalcurves}.
\end{rem}

\begin{rem}\label{lem:nox0}
Let $x \in Z^-$ and $[a,b]\subset [0,1]$ be such that $x\vert_{|a,b]}$ is a geodesic with $x\big(\left]a,b\right[\big) \subset \Omega$, and with $x(a),x(b) \in \partial \Omega$. Since $x$ is a $\mathcal V^-$--critical curve, if $[a,b] \subset\left]0,1\right[$, the derivative $\dot x$ is tangent to $\partial \Omega$ both at $s=a$ and at $s=b$. If either $a=0$ or $b=1$ (but not both), the derivative $\dot x$  is tangent to $\partial\Omega$ at one endpoint, and  orthogonal to $\partial\Omega$ at the other endpoint. Now, if $\widetilde x$ is the affine reparameterization of $x\vert_{[a,b]}$ in $[0,1]$, we have
\[
\mathcal F(\widetilde x)=(b-a)\int_a^bg(\dot x, \dot x)\,\mathrm ds \leq M_0^2
\]
because $\int_a^bg(\dot x, \dot x)ds\leq \mathcal F(x)\leq M_0^2$. Therefore, $\widetilde x$ belongs to $\mathcal G_{M_0^2}$ (cf. \eqref{eq:GL}), and so by \eqref{eq:nonsat}, there exists $x_0 \in \Omega$ and $\rho_0>0$ such that 
\begin{equation}\label{eq:Brho}
\overline{B_{\rho_0}(x_0)} \subset \Omega,
\end{equation}
and
\begin{equation}\label{eq:Z0-nointers}
x\big([0,1]\big) \cap \overline{B_{\rho_0}(x_0)}=\emptyset,\quad \forall\,x \in Z^-,
\end{equation}
where $B_{\rho_0}(x_0)=\big\{y \in M: d(y,x_0)<\rho_0\big\}$.
\end{rem}

Now, choose $\delta_1 \in\left]0,\delta_0\right[$ such that
\begin{equation}\label{eq:delta1}
\big\{y\in \overline{\Omega}:\phi(y)\geq-\delta_1\big\} \cap  \overline{B_{\rho_0}(x_0)}=\emptyset.
\end{equation}

\begin{rem}\label{eq:sottolivellibassi}
Note that,  by Corollary \ref{smallstrip}, we  have:
\begin{equation}\label{eq:delta1-bis}
\mathcal F(x) \leq \frac{\delta_1^2}{K_0^2}\quad \Longrightarrow\quad x\big([0,1]\big) \cap  \overline{B_{\rho_0}(x_0)}=\emptyset. 
\end{equation}
\end{rem}

\subsection{\texorpdfstring{$\mathcal V^-$}{V-}--Palais-Smale sequences}. 
Let us first recall  the notion of $\mathcal V^-$--Palais-Smale sequences for the functional $\mathcal F$ at any level $c>0$.
\begin{defin}\label{def:OpenPS}
Let $c>0$. 
We say that $(x_n)_n\subset \mathfrak M_\sigma$
 is a $\mathcal V^-$-\emph{Palais-Smale sequence}  for $\mathcal F$ at the level $c$ if 
\begin{equation}\label{eq:Fxnbounded}
\lim_{n\to\infty}\mathcal F(x_n) = c,
\end{equation}
and if for all $n\in\mathds N$
and for all $V_n \in \mathcal V^-(x_n)$
such that $\Vert V_n\Vert_*=1$, the following holds:
\begin{equation}\label{eq:OPS}
\mathrm d\mathcal F(x_n)[V_n] \geq -\varepsilon_n,
\end{equation} 
where $\varepsilon_n$ is a sequence of positive numbers with $\lim\limits_{n\to\infty}\varepsilon_n=0$.
\end{defin}
Given any sequence $(x_n)$ in $\mathfrak M$ with $\mathcal F(x_n)$ bounded, then standar arguments show the existence of a subsequence of $x_n$ which is uniformly convergent. For $\mathcal V^-$-Palais--Smale sequences, a stronger convergence property holds.
\begin{prop}\label{prop:convforte}
Let $(x_n)_n\in \mathfrak M$ be a $\mathcal V^-$--Palais-Smale sequence at the level $c>0$ which is uniformly convergent to a curve $x \in \mathfrak M$. Then  $x_{n}$ is strongly $H^1$-convergent to $x$.
\end{prop}
\begin{proof} 
See \cite[proof of Proposition 4.2]{OT}. 
\end{proof}

\subsection{Extension of \texorpdfstring{$\mathcal V^-$}{V-}--fields.}
Let us use the following notation: for all $x \in \mathfrak M$ and all $\rho > 0$ set
\[ 
 B(x,\rho)=\big \{z \in \mathfrak M: \dist_*(z,x)  < \rho\big\}
\]
and 
\begin{equation}\label{eq:defUxrho}
\mathcal U_\rho(x)=B(x,\rho) \bigcup B(\mathcal Rx,\rho),
\end{equation}
which is clearly $\mathcal R$--invariant. 

We set:
\begin{multline}\label{eq:V-delta}
{\mathcal W^{-}}(x) =\Big\{V \in \mathcal V^{-}(x):
g\big(\nabla \phi(x(0)),V(0))=g\big(\nabla \phi(x(1)),V(1))=0, \\
g\big(\nabla \phi(x(s)),V(s)\big) < 0\ \text{ if }  s \in ]0,1[ \text{ and }\phi\big(x(s)\big)=0\Big\}.
\end{multline} 

We have the following local property  

\begin{prop}\label{thm:constrcampiV-}
Let $x \in \mathfrak M$   and let $\mu>0$ be fixed;  assume that 
there exists $V\in \mathcal V^-(x)$ such that 
\begin{equation}\label{eq:V-discesa}
\int_0^1 g\big(\dot x, \Dds V\big)\,\mathrm ds\leq -\mu\Vert V \Vert_*.
\end{equation}
Then, for any $\epsilon >0$ there exist 
$\rho_x = \rho_x(\epsilon) > 0$ and a $C^{1}$-vector field $V_x$ defined in $\mathcal U_{\rho_x}(x)$,  such that:
\begin{itemize}
\item[(i)] $V_x(\mathcal R z)=\mathcal R V_x(z)$;\smallskip

\item[(ii)] $V_x(z) \in {\mathcal W^{-}(z)}$;\smallskip

\item[(iii)] $\Vert V_x(z) \Vert_*=1$;\smallskip

\item[(iv)] $\int_0^1 g\big(\dot z, \Dds V_x(z)\big)\,\mathrm ds \leq -{\mu}+\epsilon$,
\end{itemize}
for all $z\in \mathcal U_{\rho_x}(x)$. 
\end{prop}
\begin{proof}
Analogous to \cite[Proposition 4.3]{OT}.
\end{proof}

\begin{rem}\label{rem:perH-} Using Propositions \ref{prop:convforte} and 
\ref{thm:constrcampiV-} we see that if $(x_n)$ is a $\mathcal V^-$--Palais-Smale sequence at the level $c> 0$, then $(x_n)$ has a subsequence which is  strongly  convergent to a $\mathcal V^-$--critical curve.
\end{rem}

Using a locally Lipschitz continuous partition of the unity, we obtain:
\begin{prop}\label{V-onglobal}
Let $C \subset \mathfrak M$ be an $\mathcal R$--invariant closed set that does not contain $\mathcal V^{-}$--critical curves, and such that  $\mathcal F(C) \subset \big[0,M_0^2\big]$. Then, there exists $\mu_C>0$,  a locally Lipschitz continuous vector field $W$ defined on $C$ such that 

\begin{itemize}
\item[(i)] $W(\mathcal R z)=\mathcal R W(z)$;\smallskip

\item[(ii)] $W(z) \in {\mathcal W}^{-}(z)$;\smallskip

\item[(iii)] $\Vert W(z) \Vert_*\leq 1$;\smallskip

\item[(iv)] $\int_0^1 g\big(\dot z, \Dds W(z)\big)\,\mathrm ds \leq -\mu_C$.
\end{itemize}
 \end{prop}
\begin{proof}
As $C$ does not contain $\mathcal V^{-}$--critical curves, and   $\mathcal F(C) \subset [0,M_0^2]$,  from Remark~\ref{rem:perH-} we deduce the existence of $\mu_C> 0$ such that for every $x \in C$ there exists $V\in \mathcal V^-(x)$ satisfying  $\int_0^1 g\big(\dot x, \Dds V\big)\,\mathrm ds\leq -2\mu_C\Vert V \Vert_*$.

Then, for all $x \in C$, we can  take $\rho_x$ and $\delta_x$ as in 
Proposition~\ref{thm:constrcampiV-}, and we consider the vector field $V_x$,  defined in $\mathcal U_{\rho_x}(x)$, and satisfying (i)---(iv) of Proposition~\ref{thm:constrcampiV-} 
(in (iv), $-\mu+\epsilon$ is replaced by $-\mu_C$).

Consider the open covering $\big\{{\mathcal U}_{\rho_x}(x)\big\}_{x\in C}$ of $C$.
Since $C$ is (a metric space, hence) paracompact, there exists a locally finite open refinement
$(\mathcal A_i)_{i\in J}$, 
with $\mathcal A_i\subset{\mathcal U}_{\rho_{x_i}}(x_i)$ for all $i$, and $\bigcup\limits_{i\in J}\mathcal A_i \supset C$.
We can assume $\mathcal A_i=\mathcal R$-invariant for all $i$ (otherwise, replace $\mathcal A_i$ with $\mathcal A_i\cup\mathcal R\mathcal A_i$).
Define, for  $z \in C$ and $i \in J$:
\[
\varrho_{x_i}(z) =  \dist_{*}\big(z,C \setminus 
\mathcal A_i\big)
\]
Since $C$ and $\mathcal A_i$ are $\mathcal R$--invariant, and   $\Rcal \circ \Rcal$ is the identity map, we get
\begin{equation*}
\varrho_{x_i}(\Rcal z) = \varrho_{x_i}(z).
\end{equation*}
Finally set, for any $i \in J$,
\[
\beta_{x_i}(z) =
\frac{\varrho_{x_i}(z)}{\sum\limits_{j\in J}\varrho_{x_j}(z)},
\]
which satisfies:
\begin{equation*}
\beta_{x_i}(\Rcal z) = \beta_{x_i}(z)
\end{equation*}
and
\[
\sum_{j\in J}\beta_{x_j}(z) = 1.
\]
The desired vector field is defined by:
\[
W(z)= \sum_{j\in J}\beta_{x_j}(z)\,V_{x_j}(z).
\]
Note that this is well defined, since for all $j$, $V_{x_j}(z)\in T_z\mathfrak M$.
\end{proof}
\section{\texorpdfstring{${\mathcal V}^+$}{V-}--vector fields and the invariant set}\label{sec:V+criticalcurves}
 \subsection{The invariant set}\label{sub:invariant}
Let $\delta_1$  be as defined in \eqref{eq:delta1}. 
Since $Z^-\subset \mathcal F^{-1}\big([0,M_0^2]\big)$, by the strong concavity assumption there exists $\Delta_*>0$ such that, for any $z\in Z^-$, there exists an interval $[a,b]\subset\left]0,1\right[$ such that 
\begin{equation}\label{eq:Z-properties}
b-a\ge2\Delta_*,\; \phi\big(z(a)\big)=\phi\big(z(b)\big)=-\tfrac13{\delta_1},\; \phi\big(z(s)\big)>-\tfrac13{\delta_1} \text{ for all }s\in\left]a,b\right[.
\end{equation}

For any  $[a,b] \subset\left ]0,1\right[$ and $x \in \mathfrak M$ set:
\begin{equation}\label{eq:effe} 
f\big(x,[a,b]\big)= \min\{\phi(x(s)): s \in [a,b]\}.
\end{equation}
and for any $\rho>0$ set
\begin{equation}\label{eq:outsideball}
\mathcal B_\rho =\big\{x \in \mathfrak M: x\big([0,1]\big)\cap B_{\rho}(x_0)=\emptyset\big\} 
\end{equation}
where $x_0$ is defined at Remark \ref{lem:nox0}.
\medskip

For $x\in\mathfrak M$, let us consider intervals $[a,b] \subset\left ]0,1\right[$ with: 
\begin{equation}\label{eq:b-aDelta*}
b-a\ge\Delta_*
\end{equation} 
such that 
\begin{equation}\label{eq:agliestremi}
\phi(x(a))\geq-\frac{\delta_1}4,\quad  \phi(x(b))\geq-\frac{\delta_1}4, 
\end{equation}
and
\begin{equation}\label{eq:sopralivelli}
f\big(x,[a,b]\big)\geq-\frac{\delta_1}2.
\end{equation}
\medskip
For any $x \in \mathfrak M$  we  set:
\begin{equation}\label{eq:mathcalI}
\mathcal I(x)=\big\{[a,b] \subset\left ]0,1\right[: \text{it satisfies }\eqref{eq:b-aDelta*}-\eqref{eq:sopralivelli}
\text{ and is maximal with respect to these properties}
\big\},
\end{equation}
(with $\mathcal I(x)=\emptyset$ if no interval $[a,b]$ satisfies \eqref{eq:b-aDelta*}--\eqref{eq:sopralivelli}).

The desired invariant is defined as follows:
\begin{equation}\label{eq:Lambda*}
\Lambda_*=\Big\{x \in \mathcal F^{-1}\big([0,M_0^2]\big): x([0,1])  \cap \mathcal B_{\rho_0} = \emptyset \text{ and  }\mathcal I(x)\neq\emptyset \Big\}.
\end{equation}
where $\rho_0$ is defined as in Remark~\ref{lem:nox0}, 
and the \emph{entrance set} will be 
\begin{equation}\label{eq:Gamma*}
\Gamma_*=\Big\{x\in \Lambda_*: f\big(x,[a,b]\big)= -\tfrac12{\delta_1}\text{ for all  }[a,b] \in \mathcal I(x)  \Big\}.
\end{equation}
\begin{rem}\label{rem:Z-incl}
Note that $\Lambda_*$  is closed and $Z^-$ is contained in the interior of  $\Lambda_*$. 
\end{rem}
We will now construct a flow which leaves $\Lambda_*$, for which $\Gamma_*$ is the entrance set of $\Lambda_*$, and with the property that the functional $\mathcal F$ is strictly decreasing along the flow near $\Gamma_*$.
\subsection{\texorpdfstring{$\mathcal V^+$}{V-}--criticality}
We will now introduce a suitable notion of criticality linked to the definition of the invariant set   whose related variations drive away from the set $\{\phi \geq -\frac{\delta_1}2 \}$.

Given $x\in \mathfrak M$ consider the closed convex cone $\mathcal V^+(x)$ of $H^1\big([a,b],
\mathds R^{2N}\big)$ defined by:
\begin{multline}\label{eq:defV+(x)globale}
\mathcal V^+(x)=\Big\{V\in  \mathcal V^-(x): V(s)=0 \text{ if }\phi\big(x(s)\big)\geq-\tfrac13{\delta_1} \text{ or }  \phi\big(x(s)\big)\leq-\delta_1, \\ 
 g\big(V(s),
\nabla\phi\big(x(s)\big)\big)\ge0 \text{ if }\phi(x(s))= -\tfrac{\delta_1}2        \Big\},
\end{multline}
see Figure~\ref{fig:2}.
\begin{figure}
\centering
\includegraphics[scale=.5]{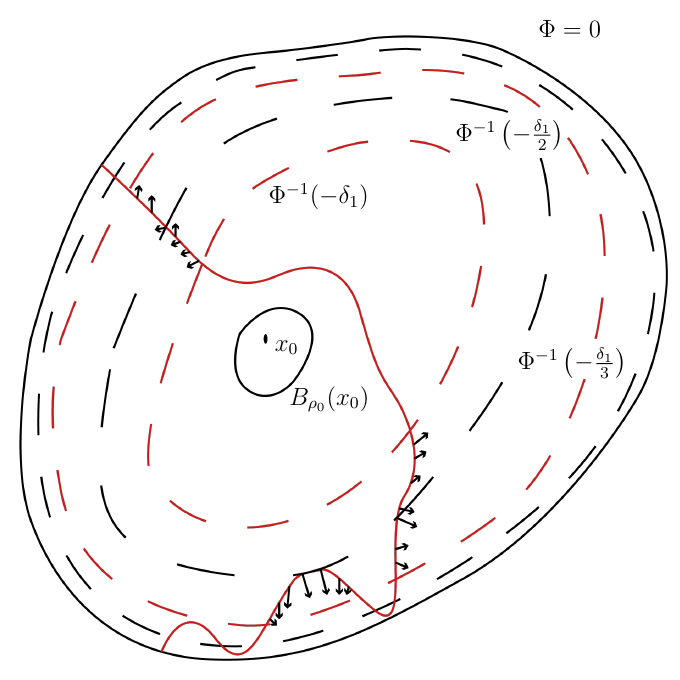}
\caption{A typical $\mathcal V^+$-field along a curve $x\in\Gamma_*$.}
\label{fig:2}
\end{figure}
Moreover, for $[a,b] \subset [0,1]$  set:  
\begin{equation}\label{eq:defV+(x)}
\mathcal V^+_{[a,b]}(x)=\Big\{V\in \mathcal V^+(x) : V(s)=0 \text{ for any } s \not \in ]a,b[\Big\}. 
\end{equation}

As for $\mathcal V^-$ the relative criticality notion goes as follows:
\begin{defin}\label{thm:defvariatcrit+}
Let $x\in\mathfrak M$ and $[a,b]\subset\left]0,1\right[$. We say that $x$ is a \emph{ $\mathcal V^+$--critical curve} if 
\begin{equation}\label{eq:varcriticality+}
\phantom{ \quad\forall\,V\in\mathcal V^+(x).}
\int_0^1 g\big(\dot x,\Dds V\big)\,\mathrm
ds\ge0, \quad\forall\,V\in\mathcal V^+(x),
\end{equation}
while $x_{|[a,b]}$ is a \emph{ $\mathcal V_{[a,b]}^{+}$--critical curve}  for $\mathcal F$ if  
\begin{equation}\label{eq:varcriticalityab+}
\phantom{ \quad\forall\,V\in\mathcal V^+_{[a,b]}(x).}
\int_a^b g\big(\dot x,\Dds V\big)\,\mathrm
ds\ge0, \quad\forall\,V\in\mathcal V^+_{[a,b]}(x).
\end{equation}
\end{defin}

We have the following cricial result. 

\begin{lem}\label{lem:nocriticality} Let $x \in \mathfrak M$ and $[a,b]\subset [0,1]$ be such that 
\begin{equation}\label{eq:agliestremi_2}
\phi(x(a))=\phi(x(b)) \in\left ]-\tfrac12{\delta_1},-\tfrac13{\delta_1}\right],
\end{equation}
\begin{equation}\label{eq:all'interno}
-\tfrac12{\delta_1} \leq \phi(x(s)) < \phi(x(a)) \text{ for any }s \in\left ]a,b\right[. 
\end{equation}
Then   $x\vert_{[a,b]}$ is  not 
 a $\mathcal V_{[a,b]}^+$--critical curve. 
\end{lem}
\begin{proof} Suppose by contradiction that $x_{|[a,b]}$ is  
a $\mathcal V_{[a,b]}^+$--critical curve.  
Then 
\begin{equation}\label{eq:alphabetacrit}
\begin{aligned}
&\int_{a}^{b} g\big(\dot x,\Dds V\big)\,\mathrm ds\ge0,  \\
\text{ for all }V\in T_x H^1&\big([a,b]\big)\ \  \text{satisfying}\ \  V(a)=V(b)=0, \text{ and }\\
g\big(V(s),\nabla&\phi\big(x(s)\big)\big)\geq 0 \text{ if }\phi\big(x(s)\big)=-\tfrac12{\delta_1}.
\end{aligned}
\end{equation}

Then by Proposition \ref{prop:descrizioneV^-}, with $\overline \Omega$ replaced by  $\phi^{-1}\big(\left[-\tfrac12\delta_1,0\right]\big)$,  we have that $x\vert_{[a,b]}$ is a geodesic with  obstacle. But 
$\{\phi \geq -\frac{\delta_1}2 \}$ is (strongly)  convex because of  strong concavity assumption, therefore, by \eqref{eq:8.1bis} (always with $\overline \Omega$ replaced by  
$\{\phi \geq  -\frac{\delta_1}2\}$) we obtain that   $x_{|[a,b]}$ is a geodesic.

But this is in contradiction with \eqref{eq:agliestremi_2}, \eqref{eq:all'interno} 
 and the strong concavity assumption (cf. Remark \ref{rem:1.4bis}).
\end{proof}

\subsection{\texorpdfstring{$\mathcal V^+$}{V-}--Palais-Smale sequences}

We will need the following version of the Palais--Smale condition with respect to the notion of 
\emph{$\mathcal V^{+}$--criticality}. Set 

\begin{equation}\label{eqmathcalIhat}
{\mathcal I}_*(x)=\big\{[a,b] \subset [0,1]: \phi\big(x(a)\big) > -\tfrac{\delta_1}{3},\ \   \phi\big(x(b)\big) > -\tfrac{\delta_1}{3},\ \ f(x,[a,b])\ge -\delta_1\big \}.
\end{equation}

\begin{lem}\label{lem:PSconvergenza}
There exist  $\mu_*>0$ and $\delta_* \in\left ]0,\frac{\delta_1}{12}\right[$ 
such that, for any $x \in \mathcal F^{-1}([0,M_0^2]$ and for any $[a,b] \in \mathcal I_*(x)$ satisfying
\begin{equation}\label{eq:defdelta*}
f\big(x,[a,b]\big) \in \big[-\tfrac12{\delta_1}-2\delta_*,-\tfrac12{\delta_1}+2\delta_*\big]
\end{equation}
 there exists  a vector  field $V$ along $x$ and $[\alpha,\beta] \subset [a,b]$ such that:  
\begin{itemize}
\item[(1)] $ \phi(x(s)) \leq -\frac{5}{12}\delta_1$ 
\footnote{
Here $-\frac{5}{12}\delta_1$ is just any fixed real number in the interval $\left]-\tfrac12{\delta_1}, -\tfrac13{\delta_1}\right[$.
}
for all $s \in [\alpha,\beta]$, $\phi(x(\alpha))=\phi(x(\beta))=-\frac{5}{12}\delta_1$ and \\  $f\big(x,[\alpha,\beta]\big)=f\big(x,[a,b]\big)$  ;\smallskip
\item[(2)] $V \in \mathcal V^+_{|[\alpha,\beta]}(x)$;\smallskip
\item[(3)] $\Vert V \Vert=1$;\smallskip
\item[(4)] $\int_\alpha^\beta g\big(\dot x, \Dds V\big)\,\mathrm ds \leq -\mu^*$.
\end{itemize}
\end{lem}

\begin{proof}
Assume by contradiction the existence of    $x_n \in \mathcal F^{-1}\big([0,M_0^2]\big)$ and $[a_n,b_n] \in \mathcal I_*(x_n)$ satisfying
\begin{equation}\label{eq:versodelta*}
f\big(x_n,[a_n,b_n]\big)\longrightarrow -\tfrac12{\delta_1},\quad\text{as $n\to\infty$},
\end{equation} 
such that  for any interval $[\alpha_n, \beta_n] \subset [a_n,b_n]$ with:
\footnote{note that \eqref{eq:intervals} is property (1) of the statement with $[\alpha_n,\beta_n]$ and $[a_n,b_n]$ replacing $[\alpha,\beta]$ and $[a,b]$ respectively.}
\begin{equation}\label{eq:intervals}
\begin{aligned}
 \phi(x_n(s)) \leq -\tfrac5{12}{\delta_1}\ &\text{ for any }s \in [\alpha_n,\beta_n], \phi(x(\alpha_n))=\phi(x(\beta_n))=-\tfrac5{12}{\delta_1} \\  &\text{ and }f\big(x_n,[\alpha_n,\beta_n]\big)=f\big(x_n,[a_n,b_n]\big)
\end{aligned}
\end{equation}
and for any vector field  $V_n \in \mathcal V^+_{[\alpha_n,\beta_n]}(x_n)$,  it is 
\begin{equation}\label{eq:noPS}
\int_{\alpha_n}^{\beta_n} g\big(\dot x_n, \Dds V_n\big)\,\mathrm ds \geq -\tfrac1n\Vert V_n\Vert.
\end{equation}
Since $\mathcal F(x_n)$ is bounded,   
up to taking a subsequence of $(x_n)$, we have the existence of $x \in F^{-1}\big([0,M_0^2]\big)$ such that 
\[
x_n \text{ is weakly convergent to $x$ in }H^{1}\big([0,1],M\big),
\]
and 
\[
x_n \to x \text{ uniformly in }[0,1].
\]

Now by \eqref{eq:versodelta*} and \eqref{eq:intervals}, up to taking a subsequence, we have:
\begin{equation}\label{eq:convergenza}
\alpha_n \to \alpha, \beta_n \to \beta, \text{ and } \alpha < \beta.
\end{equation}

Our goal is to prove that $x\vert_{[\alpha,\beta]}$ is a ${\mathcal V}^+_{[\alpha,\beta]}$--critical curve, namely 
\begin{equation}\label{eq:Gammacrit}
\int_
\alpha^\beta g\big(\dot x,\Dds V\big)\,\mathrm ds \geq 0,\ \  \text{ for all }V \in 
{\mathcal V}^+_{[\alpha,\beta]}(x),
\end{equation}
getting a contradiction with Lemma \ref{lem:nocriticality}, because $f\big(x,[\alpha,\beta]\big)=-\frac{\delta_1}2$.

Clearly, up to an affine reparameterization,  we can assume
\[
\alpha_n=\alpha,\ \text{and}\  \beta_n=\beta, \text{ for $n$ sufficiently large}.
\]
First, let us prove that 
\begin{equation}\label{eq:strongconv[a,b]}
\dot x_{n}\big\vert_{[\alpha,\beta]} \longrightarrow \dot x\big\vert_{[\alpha,\beta]}\  \text{in $L^2$, \ \ as $n\to\infty$}.
\end{equation}  

Towards this goal, we consider the exponential map $\overline{\exp}$ induced by the metric $\bar g$ described in Remark~\ref{rem:totgeo} (recall that the hypersurfaces $\phi^{-1}(-\delta)$ are $\overline g$-totally geodesic for $\delta>0$ small), and we let $c_n$ be a $H^1$--vector field along $x_n(s)$ such that
\begin{itemize}
\item $(\overline{\exp})^{-1}_{x_{n}(a)}(x(\alpha))+c_n(\alpha)=0$,
\item $(\overline{\exp})^{-1}_{x_{n}(\beta)}(x(\beta))+c_n(\beta)=0$,
\item $c_n \to 0$ in $H^1\big([\alpha,\beta]\big)$ as $n\to\infty$.
\end{itemize}
Note that such a $c_n$ exists because $x_n(\alpha) \to x(\alpha)$ and $x_n(\beta) \to x(\beta)$ as $n\to\infty$.

Consider now the $H^1$--vector field along $x_n$ given by 
\[
W_n(s)=\begin{cases} 
(\overline{\exp})^{-1}_{x_{n}(s)}(x(s))+c_n(s), \text{ if }s \in [\alpha,\beta];\\[.2cm]
0, \text{ if }s \not \in [\alpha,\beta].
\end{cases}
\]
Observe that $W_n$ is well defined for any $n$ sufficiently large, because $x_n$ tends to $x$ uniformly as $n\to\infty$. Moreover,  $W_n(\alpha)=W_n(\beta)=0$,
$W_n \to 0$ uniformly, and $\Vert W_n\Vert_*$ is bounded. 
 
Since $x_n \to x$ uniformly, by \eqref{eq:intervals},  there exists $n_0$ and  $\tilde\alpha$, $\tilde\beta$, with $\alpha< \tilde \alpha < \tilde \beta < \beta$, such that, for all $n \geq n_0$,  
 \begin{equation}\label{eq:lontananza}
-\tfrac12{\delta_1}  < \inf\Big\{\phi\big(x_n(s)\big): s \in [\alpha, \tilde \alpha]\cup[\tilde\beta, \beta]\Big\}. 
\end{equation}
Choose a piecewise affine map $\chi\colon[\alpha,\beta]\to[0,1]$ such that $\chi(\alpha)=\chi(\beta)=0$   and $\chi(s)=1$ if $s \in [\tilde \alpha,\tilde \beta]$.   Moreover, take 
\[
\lambda_n=\sup\Big\{g\big(W_n(s),\nabla\phi(x_n(s))\big)^-: s \in [\alpha,\beta]\text{ and  } \phi(x_n(s))=  -\tfrac12{\delta_1}\Big\}
\]
where $\theta^{-}$  denotes the negative part   of $\theta$. Then 
the vector field
\[
V_n(s)=W_n(s)+\lambda_n\,\chi(s)\,\nabla \phi\big(x_n(s)\big) 
\]
is in $\mathcal V^+_{[\alpha,\beta]}(x_n)$. Now $V_n$ is bounded in $H^1$, while $[\alpha_n,\beta_n]=[\alpha,\beta]$ satisfies
 \eqref{eq:noPS}. Then we have 
\begin{equation*}
\liminf_{n\to+\infty}\int_{\alpha}^{\beta} g\big(\dot x_n,\Dds V_n\big)\,\mathrm ds \geq 0.
\end{equation*}
Since $W_n \to 0 $ uniformly it is 
\[
\lambda_n \to 0,
\]
so
\begin{equation}\label{eq:disdebole}
\liminf_{n\to+\infty}\int_{\alpha}^{\beta} g\big(\dot x_n,\Dds W_n\big)\,\mathrm ds \geq 0.
\end{equation}

Now if  $U_1,\ldots,U_k$ are the domains of local charts covering $x\big([\alpha,\beta]\big)$, and  $[a_i,b_i]$ ($i=1,\ldots,k$) are intervals covering $[\alpha,\beta]$  such that $x\big([a_i,b_i]\big) \subset U_i$ for all $i$. Using the fact that $x_n$ tends to $x$ uniformly,  and $\dot x_n$ is bounded in $L^2$, one sees easily that, using the above local charts, in any interval $[a_i,b_i]$ the covariant derivative $\Dds W_n$ is given by an expression of the form:
\begin{equation}\label{eq:x-xn}
\Dds W_n = \dot x-\dot x_n+w_n^i,
\end{equation}
where $w_n^i$ is $L^2$-convergent to $0$. Then  by \eqref{eq:disdebole}
\[
\liminf_{n \to +\infty}\int_{a_i}^{b_i}g(\dot x_n, \dot x- \dot x_n)\,\mathrm ds \geq 0,
\]
(recall that $c_n\to 0$ in $H^1$).
Moreover, by the weak $L^2$--convergence of $\dot x_n$ to $\dot x$, we have:
\[
\int_{a_i}^{b_i}g(\dot x, \dot x- \dot x_n)\,\mathrm ds \to 0, 
\]
and one obtains  the $H^1$-convergence of $x_n$ to $x$.

\medskip

Now fix $V\in \mathcal V_{[\alpha,\beta]}^+(x)$ and consider 
\[ 
W_n(s)=
\big(\mathrm d\,\overline\exp_{x_n(s)}(w_n(s))\big)^{-1}\big(V(s)\big),
\]
where $w_n(s)$ is defined by the relation:
\[
\overline\exp_{x_n(s)}\big(w_n(s)\big)=x(s),\quad \forall\, s.
\]
Let $\chi$ be  as above and take
\[
\tilde \lambda_n=\sup\big\{g(W_n(s),\nabla\phi(x_n(s))^-: s \in [\alpha,\beta]\text{ and }\phi(x_n(s))=-\tfrac12{\delta_1}\big \}
\]
Since $V\in \mathcal V_{[\alpha,\beta]}^+(x)$, and $x_n \to x$ uniformly,  we have $\tilde\lambda_n \to 0$ as $n\to\infty$. Define
\[
V_n(s)=W_n(s)+\tilde\lambda_n\, \chi(s)\,\nabla \phi\big(x_n(s)\big) .
\]
Now $V_n \in \mathcal V_{[\alpha,\beta]}^+(x_n)$ and it  is bounded in $H^1$, while $[\alpha_n,\beta_n]=[\alpha,\beta]$ satisfies \eqref{eq:noPS}.  Then,  we have
\[
\liminf_{n \to +\infty}
\int_{\alpha}^{\beta} g\big(\dot x_n, \Dds V_n\big)\,\mathrm ds \geq 0.
\]
Since $x_n \to x$ in $H^1\big([\alpha,\beta]\big)$, we have  $\Dds V_n \to \Dds V$ in $L^2\big([\alpha,\beta]\big)$. 
Hence, from the above inequality we obtain \eqref{eq:Gammacrit},
and the proof is complete.
\end{proof}

Let $\delta_*$ be given as in Lemma~\ref{lem:PSconvergenza}; let us define:
\begin{multline}\label{eq:B*}
{\mathcal B}_*=\{x \in \mathcal F^{-1}\big([0,M_0^2]\big): x([0,1])  \cap \mathcal B_{\rho_0} = \emptyset \\
\text{ and  there exists }[a,b] \in \mathcal I_*(x) \text{ satisfying } \eqref{eq:defdelta*} \}.
\end{multline}
Note that $\Gamma_* \subset {\mathcal B}_*$.

\subsection{Extension of \texorpdfstring{$\mathcal V^+$}{V-}-fields} 
 
In order  to construct the global vector fields  that we shall use in        $\Lambda^*$ we first need a a local extension result about  vector fields in the class $\mathcal V^+$ for any $x$ in a neighborhood of $\Lambda_*$.

Towards this goal it will be also useful to define  the following class of vector fields in $\mathcal V^+(x)$:

\begin{equation}\label{eq:V+sigmalambda}
\mathcal W^{+}(x)=\Big\{V\in \mathcal V^+(x):     
g\big(\nabla\phi(x(s)),V(s))\big) > 0
\text{ if } 
\phi\big(x(s)\big) = -\tfrac12{\delta_1}\Big\}.   
\end{equation}

Let $\mathcal U_\rho(x)$ be defined as in \eqref{eq:defUxrho}. 
We have the following result concerning the local extension of  vector fields in $\mathcal V^{+}(x)$.  

\begin{prop}\label{thm:constrcampiV1-} Let  $\mu_*$ and ${\mathcal B}_*$  be given as in Lemma \ref{lem:PSconvergenza} and in \eqref{eq:B*} respectively.
For any  $x\in {\mathcal B}_*$   there exists $\rho_x>0$ and  a $C^1$ vector field $V_x$ defined on $\overline{\mathcal U_{\rho_x}(x)}$ such that,  for any $y \in \overline{\mathcal U_{\rho_x}(x)}$  
we have
 \begin{itemize}
\item[(i)] $V_x(\mathcal R y)=\mathcal R V_x(y)$;\smallskip

\item[(ii)] $V_x(y) \in \mathcal W^{+}(y)$;\smallskip

\item[(iii)] $\Vert V_x(y) \Vert=1$;\smallskip

\item[(iv)] $\int_{0}^{1} g\big(\dot y, \Dds V_x(y)\big)\,\mathrm ds \leq -\tfrac12{\mu_*}$.
\end{itemize}
\end{prop}

\begin{proof}
Assume first $x \not= \mathcal R x$. Let $[a,b]\in \mathcal I_*(x)$ satisfy \eqref{eq:defdelta*}, and let $V$ and $[\alpha,\beta]$ be given as in Lemma \ref{lem:PSconvergenza}; 
Consider  the exponential map as discussed in Remark \ref{rem:totgeo}. 

For a fixed $\rho>0$ sufficiently small 
and any $y \in B(x,\rho)$, set:
\begin{equation}\label{eq:defsuz} 
W(y)(s)=
\big(\mathrm d\,\overline\exp_{y(s)}(w(s))\big)^{-1}\big(V(s)\big),
\end{equation}
where $w(s)$ is defined by the relation:
\begin{equation}\label{eq:expbar}
\overline\exp_{y(s)}\big(w(s)\big)=x(s),\quad \forall\, s.
\end{equation}
If $\rho$ is sufficiently small, there exists an interval $[\tilde\alpha,\tilde \beta] \subset\left ]\alpha,\beta\right]$ such that $\phi\big(y(s)\big)$ is bounded away from $-\frac12{\delta_1}$ for all $s \in [\alpha,\tilde\alpha] \cup [\tilde\beta,\beta]$.

Let $\chi$ be a piecewise affine map such that 
$\chi(s)=0$ for any $s \not\in\left]\alpha,\beta\right[$ and 
$\chi\equiv 1$ on $[\tilde\alpha,\tilde\beta]$.
 Note that for any $\rho$ sufficiently small, $\phi\big(y(s)\big) \in\left ]-\delta_1,-\tfrac13{\delta_1}\right[$ for any $s \in \left[\alpha,\beta\right]$. 
 We want to modify $W(y)$ in order to obtain a vector field in $W^+(y)$ defined on the whole interval $\left[0,1\right[$, and to this aim we employ another piecewise affine map, that will be denoted by $\chi_x$. 
Such function must satisfy:
\begin{itemize}  
\item $\chi_x(s)>0$ on any $\left]\tilde a,\tilde b\right[$ maximal interval such that $\phi\big(x(\left]\tilde a,\tilde b\right[)\big) \subset ]-\delta_1,-\frac{\delta_1}3[$, and such that there exists $\bar s\in \left]\tilde a,\tilde b\right[$ with $\phi\big(x(\bar s)\big)=-\tfrac12{\delta}$; 
\item $\chi_x(s)=0$ if $s$ is not contained in any interval $\left]\tilde a,\tilde b\right[$ as above.
\end{itemize} 
Set
\[
\begin{aligned}
\lambda=\lambda(\rho)=\sup\big\{g(W(y)(s),&\nabla\phi(y(s)))^-: \\&
s \in [\alpha,\beta],\ \   y \in B(x,\rho),\  \text{ and }\  \phi(y(s))=-\tfrac12{\delta_1}\big\}
\end{aligned}
\]
and note that
\[
\lim_{\rho\to 0}\lambda(\rho)=0.
\]
Finally define 
\begin{equation}\label{eq:defWlambda}
W_\rho(y)(s)=W(y)(s) + \lambda(\rho)(\chi(s)+\chi_x(s))\nabla \phi(y(s)).
\end{equation}
Thanks to such a definition we have   
\[
g(W_{\rho}(y)(s),\nabla \phi(y(s)) > 0 \text{ for all }y\in B(x,\rho),
\text{ and for all }s   \text{ such that } \phi\big(y(s)\big)=-\tfrac12{\delta_1},
\]
and we have also  that  $W_{\rho}(y)\in \mathcal W^+(y)$ for all $y \in  B(x,\rho)$.

It is also easy to see that  
\begin{equation}\label{eq:limiteWl}
\lim_{\rho\to 0} \sup_{y \in B(x,\rho)}
\Vert W_\rho(y)- V \Vert=0.
\end{equation}

Since for all $\rho$ sufficiently small $W_\rho(y)\not\equiv 0$, we can define 
\[
V_\rho(y)(s)=\frac{W_\rho(y)(s)}{\Vert W_\rho(y) \Vert}.
\]
Since $\Vert V \Vert =1$, from \eqref{eq:limiteWl} we get: 
\begin{equation*}
\lim_{\rho\to 0}\  \sup_{y \in B(x,\rho)}
\Vert V_\rho(y)-V^+ \Vert_{a,b}=0.
\end{equation*} 
Then, recalling that we are assuming  $x \not= \mathcal R x$, we  extend $V_\rho$ 
to $\mathcal R\big(B(x,\rho)\big)$ by setting
\[
V_\rho(\mathcal Ry)=\mathcal R V_\rho(y).
\]
Then, the  desired vector field $V_x$ is obtained by setting $V_x=V_\rho$, where  $\rho=\rho_x$ is chosen sufficiently small so that (i)--(iv) are satisfied on $\overline{\mathcal U_{\rho_x}(x)}$.
\smallskip

Now, assume $x=\mathcal Rx$, namely: $x(s)=x(1-s)$ for any $s \in [0,1]$. If $V\in \mathcal V^+_{[\alpha,\beta]}$ is the vector field given as in Lemma \ref{lem:PSconvergenza}, we can assume $\beta=1-\alpha$, since $V=0$ outside $[\alpha,\beta]$. Now consider
\[
\widetilde V=\frac{V+\mathcal R V }{\Vert V+\mathcal R V  \Vert_*}.
\]
Since 
\[
\int_0^1g\big(\dot x, \tfrac{\mathrm DV}{\mathrm ds}\big)\,\mathrm ds=\int_0^1g\big(\dot x, \tfrac{\mathrm D}{ds}\mathcal R V\big)\,\mathrm ds
\]
we see that $V +\mathcal R V\not\equiv 0$, while $\widetilde V$ satisfies
\[
\int_0^1g\big(\dot x, \tfrac{\mathrm D\tilde V}{\mathrm ds}\big)\,\mathrm ds \leq -\mu_*.
\]
We can therefore assume that $V$ satisfies $
V = \mathcal R V$.
If  $W_\rho$ is the vector field as in \eqref{eq:defWlambda}, we can finally choose 
\[ 
V_x(y)= \frac{W_\rho(y)+ W_\rho(\mathcal R y)}{\big\Vert  W_\rho(y)+ W_\rho (R y) \big\Vert_*}
\]
and the proof is complete.
\end{proof}
As in the proof of Proposition \ref{V-onglobal}, using  again a locally Lipschitz continuous   
partition of the unity one obtains:
\begin{prop}\label{V+onC*} There exists a locally Lipschitz continuous map $W^+$ defined in ${\mathcal B}_*$ such that,  for all $x \in \mathcal B_*$ the following hold:
\begin{itemize}
\item[(i)] $W^+(\mathcal R x) =\mathcal R W^+(x)$;\smallskip

\item[(ii)] $W^+(x) \in {\mathcal W}^{+}(x)$   \smallskip

\item[(iii)] $\Vert W^+(x)(z) \Vert\leq 1$;   \smallskip

\item[(iv)]  $\int_{0}^{1} g\big(\dot x, \Dds W^+(x)\big)\,\mathrm ds \leq -\tfrac12{\mu_*}$.
\end{itemize}
\end{prop}

\section{Deformation results and proof of Theorem \ref{thm:generale} }\label{sec:CVO}

Let $\delta_1$ be as in \eqref{eq:delta1} and $K_0$ as defined at \eqref{eq:ko}. Denote by $\mathcal O$ the set of the OGCs and set
\begin{equation}\label{eq:Or}
U_r(\mathcal O)=\big\{x \in \mathfrak M: \text{ dist}_*(x,\mathcal O)<r\big\}.
\end{equation}
Let ${\mathcal B}_*$  be as in \eqref{eq:B*} and let $\delta_*$ be as in Lemma~\ref{lem:PSconvergenza}.

\subsection{A global integral flow}
Using the vector fields defined in Propositions \ref{V-onglobal} and \ref{V+onC*}, we will now prove the following existence result of a global flow, which plays the same role as the gradient flow in the classical smooth case.
\begin{prop}\label{prop:etaglobale} 
There exists  a continuous maps $\eta_*\colon\mathds{R}^+\times \mathcal F^{-1}\big([0,M_0^2]\big)\to\mathcal F^{-1}\big([0,M_0^2]\big)$ such that
for all $z\in \mathcal F^{-1}\big([0,M_0^2]\big)$, the following properties hold:
\begin{enumerate}
\item\label{itm:id1}$\eta_*(0,z)=z$;\smallskip

\item $\mathcal R$-equivariance: $\eta_*(\tau,\mathcal R z) = \mathcal R \eta_*(\tau,z)$;\smallskip

\item $\dist_*\big(\eta_*(\tau_2,z),\eta_*(\tau_1,z)\big) \leq |\tau_2-\tau_1|$;\smallskip

\item\label{itm:C*} $\Lambda_*$ is $\eta_*$-invariant, i.e., $\eta_*(\tau_0,z) \in \Lambda_*$ then  $\eta_*(\tau,z)\in \Lambda_*$ for any $\tau \geq \tau_0$;\smallskip

\item\label{itm:G*} $\Gamma_*$ is the entrance set: if $0\leq \tau_1 < \tau_2$ and $z \in \mathcal F^{-1}\big([0,M_0^2]\big)$ are such that $\eta_*(\tau_1,z) \not\in\Lambda_*$ and $\eta_*(\tau_2,z) \in \Lambda_*$, then there exists a unique $\tau_0 \geq 0$ such that $\eta_*(\tau_0,z) \in \Gamma_*$;\smallskip

\item \label{itm:discesa}  for any $r \in\left]0,1\right[$ there exists $\mu_*(r)> 0$ such that if  
\[
\eta_*(\tau,z) \in \left(
\mathcal F^{-1}
\left(\left[\frac{r\delta_1^2}{2K_0^2},M_0^2\right]\right)
\setminus\left( 
{\mathcal U}_r({\mathcal O}) \bigcup \mathrm{int}\,\Lambda_*
\right)\right) \bigcup {{\widehat{\mathcal B}}_*},
\] 
for all  $\tau \in [\tau_1,\tau_2]$, where 
\[
{\widehat{\mathcal B}}_*=\left\{
x\in{\mathcal B}_*\,:\,\exists\; [a,b]\in{\mathcal I}_*(x) \text{\ with\ } f\big(x,[a,b]\big)\in\left[-\tfrac{\delta_1}{2}-2\delta_*,-\tfrac{\delta_1}{2}+\tfrac{\delta_*}{2}\right]  
\right\}
\] 
then 
\[ 
\mathcal F\big(\eta_*(\tau_2,z)\big)\leq 
\mathcal F\big(\eta_*(\tau_1,z)\big)-\mu_*(r)(\tau_2-\tau_1);
\] 
\item\label{itm:eta_*6bis} 
$
\mathcal F\big(\eta_*(\tau_2,z)\big)\leq 
\mathcal F\big(\eta_*(\tau_1,z)\big) \text{ for all }\tau_1 \leq \tau_2$;\smallskip

\item\label{itm:cost}  
$\eta_*(\tau,x)=x$  for any $x \in \mathfrak C_0$;\smallskip

\end{enumerate}

(recall that $\mathfrak C_0$ is defined in \eqref{eq:CC0}).
\end{prop}

\begin{proof}
For any $r \in\left ]0,1\right[$ apply Proposition~\ref{V-onglobal} with 
\[
C_r=\left(
\mathcal F^{-1}
\left(\left[\frac{r\delta_1^2}{2K_0^2},M_0^2\right]\right)
\setminus\left( 
{\mathcal U}_r(\mathcal O) \bigcup \mathrm{int}\,\Lambda_*
\right)\right) \bigcup {\mathcal B}_*
\]
and denote by  $W^-_r$ the vector field that appears in its statement. 
Note that ${\mathcal B}_*$ is has positive distance from $Z^- \cup \mathcal O$, and for this reason the use of Proposition~\ref{V-onglobal} here is allowed.
Indeed, by strong concavity, $\mathcal B_*\cap(Z^-\cup\mathcal O)=\emptyset$ because for all $x\in Z^-\cup\mathcal O$ and any $[a,b]\in\mathcal I(x)$, it must be $f(x,[a,b])\ge-\tfrac{\delta_1}{3}$, whereas if $x\in\mathcal B_*$ then there exists $[a,b]\in\mathcal I_*(x)$ such that
$$
f(x,[a,b])\le -\frac{\delta_1}{2}+2\delta_*<-\frac{\delta_1}{3},
$$
recalling  from Lemma \ref{lem:PSconvergenza} that $\delta_*<\tfrac{\delta_1}{12}$.

Using a partition of the unity argument we can extend $W^-_r$ to a locally Lipschitz continuous map $V^-$ defined on all $\mathcal F^{-1}\big([0,M_0^2]\big)$  and satisfying: 
\begin{itemize}
\item $V^-(x)=0$ for all $x \in \mathcal O \cup \mathfrak C_0$; \smallskip

\item $V^-(\mathcal Rx)=\mathcal RV^-(x)$ ;\smallskip

\item $V^-(x) \in \mathcal W^-(x)$;\smallskip

\item $\Vert V^-(x) \Vert_* \leq 1$;\smallskip

\item $\int_0^1 g\big(\dot x, \Dds V^-(x)\big)\,\mathrm ds \leq - \mu(r)$,
for every $x\in C_r$, where $ \mu(r)=\mu_C>0$ is given  in Proposition~\ref{V-onglobal};\smallskip

\item $\int_0^1 g\big(\dot x, \Dds V^-(x)\big)\,\mathrm ds \leq 0$ for every $x \in  \mathcal F^{-1}\big([0,M_0^2]\big)$.
\end{itemize}\smallskip

Since $\text{int}\,\Lambda_*$ is outside $C$, we use the following construction to work on $\Lambda_*$. Set
\begin{equation*}
\mathcal V=\big\{x \in {\mathcal B}_*: \text{ there exists }[\alpha,\beta] \in \mathcal I_*(x) \text{ such that }   
  f\big(x,[\alpha,\beta]\big)>-\tfrac12{\delta_1}-\delta_*\big\},
\end{equation*}
\begin{equation*}
\mathcal U=\big\{x \in {\mathcal B}_*: \text{ there exists }[\alpha,\beta] \in \mathcal I_*(x) \text{ such that }   
  f\big(x,[\alpha,\beta]\big)>-\tfrac12{\delta_1}-\tfrac12{\delta_*}\big \}.
\end{equation*}
Note that $\overline{\mathcal U} \subset \mathcal V$.

Let $W^+$ be the vector field given by Proposition \ref{V+onC*}.
 
Choose a Lipschitz continuous map $\chi_+: \mathcal F^{-1}\big([0,M_0^2]\big) \to [0,1]$ such that:
\begin{itemize}
\item $\chi_+(x)=1$ if $x \in {\mathcal B}_*$ and there exists $[\alpha,\beta]\subset [0,1]$ such that 
\[
f\big(x,[\alpha,\beta]\big) \le-\tfrac12{\delta_1}+\tfrac12{\delta_*};\]
\item  $\chi_+(x)=0$  if $x \in {\mathcal B}_*$ and
for all  $[\alpha,\beta]\subset [0,1]$,  $f\big(x,[\alpha,\beta]\big) \geq -\tfrac12{\delta_1} + \delta_*$. 
\end{itemize}
Set $\widehat W^+=\chi_+W^+$, and trivially extend to zero $\widehat W^+$ outside $\mathcal B_*$. Now choose a Lipschitz continuous map $\chi_*: \mathcal F^{-1}\big([0,M_0^2]\big) \to [0,1]$ such that $\chi_*\equiv1$ on $\overline{\mathcal U}$ and $\chi_*(x)\equiv0$ outside $\overline{\mathcal V}$.
Finally we set 
\[
V_*(x)=\chi_*(x)\widehat W^+(x)+(1-\chi_*(x))V^-(x),
\]
and the key point here is that $V_*$ coincides with $\widehat W^+$ in a neighborhood of $\Gamma_*$. 

The homotopy $\eta_*$ is defined as the flow of $V_*$: 
\[
\begin{cases}
\dfrac{\partial \eta_*}{\partial \tau}=V_*(\eta_*),\\[.2cm]
\eta_*(0,x)=x \in  \mathcal F^{-1}\big([0,M_0^2]\big).
\end{cases}
\]

Let us observe that, due to the definitions of $\mathcal I(x)$ \eqref{eq:mathcalI} and $\mathcal V^+$ \eqref{eq:defV+(x)globale},  if $[\alpha,\beta]\in\mathcal I(x)$ it follows that, $\forall\tau\ge 0$, $\eta_*(\tau,x(\alpha))=x(\alpha)$ and $\eta_*(\tau,x(\beta))=x(\beta)$, and then $[\alpha,\beta]\in\mathcal I(\eta_*(\tau,x))$, $\forall\tau\ge 0$. 

By  Propositions~\ref{V-onglobal} and \ref{V+onC*} the flow $\eta_*$ satisfies \eqref{itm:id1}---\eqref{itm:cost} with  
\[
\mu^*(r)=\min\big\{\mu(r), \tfrac12{\mu^*}\big\}.\qedhere
\]
\end{proof}
\subsection{Deformation Lemmas}\label{sec:deflem}
Let $\delta_1$ be as in \eqref{eq:delta1} and  let $\eta_*$ be as in Proposition \ref{prop:etaglobale}. 
\begin{prop}\label{teo:deformation1} 
Let $\eta_*$ be as in Proposition \ref{prop:etaglobale}. 
Let $c \in\left [\frac{\delta_1^2}{K_0^2},M_0^2\right]$ be such that $\mathcal F^{-1}(c)$ does not contain any OGC.

Then, there exists $\varepsilon \in \left]0,\frac{\delta_1^2}{2K_0^2}\right[ $  such that 
\[
\eta_*\big(1,\mathcal F^{c+\varepsilon}\big) \subset \mathcal F^{c-\epsilon} \cup \Lambda_*.
\] 
\end{prop}

\begin{proof}
Fix  $\overline \varepsilon\in \left]0,\frac{\delta_1^2}{2K_0^2}\right[$  such that $\mathcal F^{-1}\big([c-\overline \varepsilon,c+\overline \varepsilon]\big)$  does not contain OGCs. Fix $r > 0$ such that 
$\mathcal F^{-1}\big([c-\overline \varepsilon,c+\overline \varepsilon]\big)$ does not contain  $\overline{\mathcal U_r(\mathcal O})$.

Now, let  $\mu_*(r)$ be given as in Proposition \ref{prop:etaglobale}. Then, using \eqref{itm:C*} and 
\eqref{itm:discesa} of Proposition \ref{prop:etaglobale}, we obtain the proof choosing $\varepsilon = \min\big\{\bar \varepsilon, \frac12{\mu_*(r)}\big\}$. Indeed, with such a choice for $\epsilon$, keeping into account \eqref{itm:C*} and \eqref{itm:discesa} of Proposition~\ref{prop:etaglobale}, we have that $\mathcal F(\eta_*(1,x))\leq c-\epsilon$ or there exists $\tau_0(x) \in [0,1]$ such that $\eta(\tau,x) \in \Lambda_*$ for any $\tau \geq \tau_0(x)$.
\end{proof}

Let us assume from now on that the  number of geometrically distinct OGCs is finite, say $\gamma_1,\ldots,\gamma_k$, and fix $r_*> 0$ such that 
\begin{itemize}

\item the sets  $\big\{y \in \partial \Omega : \dist(y,\gamma_i(0)) < r_* \big\}$ and  $\big\{y \in \partial \Omega : \dist(y,\gamma_i(1)) < r_*\big\}$ are disjoint and contractible in $\partial\Omega$ for all $i$;\smallskip

\item $\overline{B(\gamma_i,r_*)} \cap \overline{B(\gamma_j,r_*)} = \emptyset$ for every $i\not=j$;\smallskip

\item $\overline{B(\gamma_i,r_*)} \cap \overline{B(\mathcal R \gamma_i,r_*))} =\emptyset$ for every $i$,\smallskip

\item $\Big(\overline{B(\gamma_i,r_*)}\cup \overline{B(\mathcal R \gamma_i,r_*)}\Big) \cap \Lambda^*  =  \emptyset$ for any $i$, \smallskip

\item $\overline{U_r(\mathcal O)} \cap \mathfrak C_0=\emptyset$, \smallskip

\end{itemize}\smallskip
where $\mathcal O_r$ is defined at \eqref{eq:Or}.

Note that, for every $r \in\left]0,r_*\right]$, we have : 
\begin{equation}\label{eq:defOr}
\mathcal O_r=\bigcup_{i=1,\ldots,k} B(\gamma_i,r) \cup B(\mathcal R\gamma_i,r).
\end{equation}

Using again  Proposition~\ref{prop:etaglobale} we obtain the  following
\begin{prop}\label{teo:deformation2}
Assume there are only a finite number of OGCs, and
let $c \in\left [\frac{\delta_1^2}{K_0^2},M_0^2\right]$ be such that there exists at least an OGC $\gamma$ with $\mathcal F(\gamma)=c$. 
Let $r_*$ be as above.
Then, there exists $\varepsilon>0$   such that 
\[
\eta_*(\frac{r_*}{2},\mathcal F^{c+\epsilon} \setminus \mathcal O_{r_*}) \subset \mathcal F^{c-\epsilon} \cup \Lambda_*.
\]
where $\eta_*$ is given by Proposition \ref{prop:etaglobale}.
\end{prop}
\begin{proof}
It foillows readily from Proposition~\ref{prop:etaglobale}, arguing as in the proof of Proposition~\ref{teo:deformation1}.
\end{proof}

For our next result, assumption~\eqref{eq:nonsat} will play a crucial role.
\begin{prop}\label{prop:sublevels}
Let $\delta_1$ be as in \eqref{eq:delta1}. Under the assumptions of Theorem~\ref{thm:generale},
there exists a continuous map $H\colon[0,1] \times {\mathcal F}^{-1}\big([0,{\delta_1^2}/{K_0^2}]\big) \cup \Lambda_* \to \mathfrak M$ such that, for all $
x \in {\mathcal F}^{-1}\big([0,{\delta_1^2}/{K_0^2}]\big) \cup \Lambda_*$, the following properties hold:
\begin{itemize}
\item $H(0,x)=x$,
\item $H(\tau,x)=x$ for all $\tau$ and for all $x \in \mathfrak C_0$,
\item $H(\tau,\cdot)$ is $\mathcal R$-equivariant for all $\tau$,
\item $H(1,x)(s) \in \partial \Omega$, for all $s \in [0,1]$.
\end{itemize}
\end{prop}
\begin{proof} 
Let us fix  a homeomorphism $\Psi\colon\overline \Omega \to \mathds{D}^N$
and set 
\[
y_0=\psi(x_0)
\] 
where $x_0$ satisfies \eqref{eq:Brho} and  \eqref{eq:Z0-nointers}. 
Then, let us choose  a homotopy \[h_0\colon[0,1] \times \mathds{D}^N\setminus{y_0} \longrightarrow\mathds{D}^N\setminus{y_0} \] satisfying:
\begin{itemize}
\item $h_0(0,z)=z$ for all $z$,
\item $h_0(1,z)\in \partial\mathds D^N$ for any $z\in \mathds{D}^N\setminus \{y_0\}$.
\end{itemize}
If  $\mathcal F(x)\leq {\delta_1^2}/{K_0^2}$, or  if $x\in \Lambda_*$, then $x$ does not cross the set $\overline{B_{\rho_0}(x_0)}$
(cf. \eqref{eq:sottolivellibassi} and the definition of $\Lambda^*$). Thus, we can 
consider the homotopy $K$ defined  by:
\[
K(\tau,x)(s)=\Psi^{-1}\big(h_0(\tau,\Psi(x(s))\big).
\]
Such a map may fail to produce curves with the desired $H^1$-regularity, in view of the fact that $\Psi$ may fail to be a diffeomorphism. In this case a further broken geodesic procedure is rquired to obtain the desired map $H$, as in the proof of Lemma~\ref{thm:corde}.
\end{proof}

The topological argument that will be needed, which employs the notion of relative category of a pair of topological
spaces, requires the construction of a suitable homotopy on the set $\mathfrak C$, described in  Proposition~\ref{teo:deformation6} below. \smallskip

Let $\Dgot$ denote the family of all closed $\Rcal$--invariant subset of $\mathfrak C$. 
Given $\mathcal D\in\Dgot$ and $\tau \geq 0$, we denote by $\mathcal A_\tau$ the $\mathcal R$-invariant set:
\begin{equation}\label{eq:Ah}
{\mathcal A}_\tau=\big\{x \in \mathcal D: \eta_*(\tau,x)\in \mathcal O_{r_*}\big\}.
\end{equation}
We need the following result, whose complete proof can be found in \cite{OT}.
\begin{prop}\label{teo:deformation6}
Fix $\mathcal D\in\Dgot$, $\tau \geq 0$, and let $\mathcal A_\tau$ be as in \eqref{eq:Ah}. Then, there exists a continuous map \[k_*\colon [0,1] \times \mathcal A_\tau \longrightarrow \mathfrak C \setminus \mathfrak C_0,\] such that:
\begin{itemize}
\item[(a)] $k_*(0,x)=x$, for every $x \in \mathcal A_\tau$;
\item[(b)] $k_*(\tau,\mathcal R x)=\mathcal R k_*(\tau,x)$, for all $\tau \in [0,1]$, for all $x \in \mathcal A_\tau$;
\item[(c)] $k_*(1,\mathcal A_\tau)=\{y_0,Ry_0\}$, for some $y_0 \in \mathfrak C \setminus \mathfrak C_0$.
\end{itemize}
\end{prop} 
\begin{proof}[Sketch of the proof]
The first observation is that a homotopy that satisfies (a), (b) and (c) above, but taking values in $\mathfrak C$ (i.e., possibly having points of $\mathfrak C_0$ in its image), does exist. This follows from the fact that, by definition, the homotopy $\eta_*$ carries $\mathcal A_\tau$ 
into the set $\mathcal O_{r_*}$. For $r_*>0$ small enough, this set is retractible in $\mathfrak M$ onto the finite set 
$\big\{\gamma_1,\ldots,\gamma_k,\mathcal R\gamma_1,\ldots,\mathcal R\gamma_k\big\}$, and again this finite set is retractible
in $\mathfrak M$, say, to the two-point set $\big\{\gamma_1,\mathcal R\gamma_1\big\}$. 

Composing with the \emph{endpoints mapping}, $\mathfrak M\ni\gamma\mapsto\big(\gamma(0),\gamma(1)\big)\in\partial\Omega\times\partial\Omega\cong\mathfrak C$, one obtains a homotopy $h_*\colon[0,1]\times\mathcal A_\tau\to\mathfrak C$ that carries $\mathcal A_\tau$ to the two-point set $\{y_0,\mathcal Ry_0\}$, where $y_0=\big(\gamma_1(0),\gamma_1(1)\big)$. 

Now using $h_*$ we can construct another homotopy carrying $\mathcal A_\tau$ to the two-point set $\{y_0,\mathcal Ry_0\}$ in $\mathfrak C\setminus\mathfrak C_0$. Note that $\mathcal A_\tau\cap\mathfrak C_0=\emptyset$; namely, $\eta_*(\tau,\cdot)$ fixes all points of $\mathfrak C_0$ for all $\tau\in[0,1]$, by \eqref{itm:cost} of Proposition~\ref{prop:etaglobale}, and $\mathcal O_{r_*}\cap\mathfrak C_0=\emptyset$ for $r_*>0$ small enough. 

The second observation is that $\mathcal A_\tau$ can be written as the \emph{disjoint union} $F_1\bigcup F_2$ of two closed sets $F_1,F_2\subset\mathcal D$, with $\mathcal RF_1=F_2$; namely:
\[F_1=h_*(1,\cdot)^{-1}(y_0),\quad\text{and}\quad F_2=h_*(1,\cdot)^{-1}(\mathcal R y_0).\]
To conclude the proof, it suffices to show that $F_1$ is contractible in $\mathfrak C\setminus\mathfrak C_0$ to the singleton $\{y_0\}$ via some homotopy $\widetilde h_*\colon[0,1]\times F_1\to\mathfrak C\setminus\mathfrak C_0$. The desired homotopy $k_*$ will then be obtained by extending $\widetilde h_*$ to $[0,1]\times F_2$ by $\mathcal R$-equivariance and the map $\widetilde h_*$ can be contructed as in the proof of \cite[Proposition 5.11]{OT}.
\end{proof}

\subsection{Proof of Theorem \ref{thm:generale}}\label{sec:main}
 
Our proof of Theorem~\ref{thm:generale} will be now finalized using minimax theory  and a suitable version of Lusternik-Schnirelman relative category,
as defined in \cite[Definition 3.1]{FW}.
For all standard definitions of the relative category
and other relative cohomological indexes see for instance \cite{FH} and references therein.

\begin{defin}\label{def:10.1}
Let $X$ be a topological space and let $Y$ be a closed subset of $X$. A closed subset $F$ of $X$ has \emph{relative category
equal to $k\in \mathds{N}$}, and we write $\cat_{X,Y}(F)=k$, if $k$ is the minimal positive integer such that there exists a family $(A_i)_{i=0}^k$ of open subsets of $X$ such that $F\subset\bigcup\limits_{i=0}^k
A_i$,  $F\cap Y\subset A_0$, and such that for all $i=0,\ldots,k$ there exists continuous maps $h_i\colon[0,1]\times
A_i\to X$ with the following properties:
\begin{enumerate}
\item\label{itm:def10.1-1} $h_i(0,x)=x,\,\forall x\in A_i,\,\forall i=0,\ldots,k$; 
\smallskip

\item\label{itm:def10.1-2}
for every $i=1,\ldots,k$:
\begin{enumerate}
\item\label{itm:def10.1-2a} there exists $x_i\in X\setminus Y$ such that $h_i(1,A_i)=\{x_i\}$;
\item\label{itm:def10.1-2b} $h_i\big([0,1]\times A_i\big)\subset X\setminus Y$;
\end{enumerate}\smallskip

\item\label{itm:def10.1-3} if $i=0$:
\begin{enumerate}
\item\label{itm:def10.1-3a} $h_0(1,A_0)\subset Y$; \item\label{itm:def10.1-3b} $h_0(\tau,A_0\cap Y)\subset
Y,\,\forall\,\tau\in[0,1]$.
\end{enumerate}
\end{enumerate}
\end{defin}

For any $\Rcal$-invariant subset $X\subset \mathfrak M$, we denote by $\widetilde X$ the quotient space with respect to the equivalence relation induced by $\mathcal R$. In particular, we will consider the sets $\widetilde{\mathfrak C}$ and $\widetilde{\mathfrak C}_0$, where $\mathfrak C$ and $\mathfrak C_0$ are defined in \eqref{eq:CC0}.
For our minimax argument we will  use the relative category $\cat_{\widetilde{\mathfrak C},\widetilde{\mathfrak C}_0}(\widetilde{\mathfrak C})$. \smallskip

Using the topological properties of the $(N-1)$--dimensional real projective space, it it is not hard to show that
\begin{equation}\label{eq:10.1}
\cat_{\widetilde{\mathfrak C},\widetilde{\mathfrak C_0}}(\widetilde{\mathfrak C}) = N.
\end{equation}
Details of the proof of \eqref{eq:10.1} can be found in reference \cite{esistenza}.

\begin{proof}[Proof of Theorem \ref{thm:generale}]

Let us denote by  $\Dgot$ the class of closed $\Rcal$--invariant subset of $\mathfrak C$ and le $eta_*$ be the homotopy given in Proposition \ref{prop:etaglobale}. As usual define, for
any $i=1,\ldots,N$,
\begin{equation}\label{eq:10.2}
\Gamma_i=\Big\{\Dcal\in\Dgot\,:\,\cat_{\widetilde{\mathfrak C},\widetilde{\mathfrak C}_0}(\widetilde\Dcal)\ge i\Big\},
\end{equation}
and we set
\begin{equation}\label{eq:10.3}
c_i=\inf_{\overset{\Dcal\in\Gamma_i}{\tau \geq 0}}\sup\big\{ \mathcal F(x):{x \in \eta_*(\tau,\mathcal D) \setminus \Lambda_*}\big\}.
\end{equation}
Each $c_i$ is a  well defined real number. Namely, since  ${\mathfrak C} \in \Gamma_i$ and $\eta_*(0,x)=x$ for any $x$:
\begin{equation}\label{eq:maggiorazione}c_i\le M_0^2.\end{equation}
Clearly, since $\mathcal F \geq 0$, $c_i\geq 0$ for any $i$.  Note also that $c_i \leq c_{i+1}$.

Now, by Proposition~\ref{prop:sublevels}: 
\[ 
c_1 \geq \dfrac{\delta_1^2}{K_0^2}.
\]
Moreover, from Proposition~
\ref{teo:deformation1} it follows that for any $i\in\{1,\ldots,N\}$, there exists an orthogonal geodesic chord $\gamma_i$ such that $\mathcal F(\gamma_i)=c_i$.  Assuming that the number of OGCs is finite, from Propositions \ref{teo:deformation2} and \ref{teo:deformation6} we deduce that 
\[
\phantom{\quad \text{ for all }1\le i\le N-1,}
c_i<c_{i+1},\quad \text{ for all }1\le i\le N-1,
\]
i.e., $c_1,\ldots,c_N$ form a sequence of $N$ distinct positive real numbers. Thus, we have $N$ OGCs on which the functional $\mathcal F$ takes distinct values.
Using the transversality conditions at the endpoints, one sees easily that if $x_1$ and $x_2$ are OGSs with $x_1\big([0,1]\big)=x_2\big([0,1]\big)$, then $\mathcal F(x_1)=\mathcal F(x_2)$. This says that we have found $N$ geometrically distinct OGCs, proving Theorem~\ref{thm:generale}.
\end{proof}

\end{document}